\documentclass[11pt,english,a4paper, leqno]{article}
\usepackage{amssymb,amsfonts,amsmath,amsthm,amscd}
\usepackage[T1]{fontenc}
\usepackage[latin9]{inputenc}
\usepackage{geometry}
\usepackage{xcolor}
\usepackage{pdfcolmk}
\usepackage{refstyle}
\usepackage{textcomp}

\usepackage{esint}
\PassOptionsToPackage{normalem}{ulem}
\usepackage{ulem}
\def\IZ{{\mathbb Z}}
\def\IR{{\mathbb R}}
\def\IP{{\mathbb P}}

\def\n{\noindent}
\def\dis{\displaystyle}
\makeatletter

\AtBeginDocument{}
\RS@ifundefined{subref}
  {\def\RSsubtxt{section~}\newref{sub}{name = \RSsubtxt}}
  {}
\RS@ifundefined{thmref}
  {\def\RSthmtxt{theorem~}\newref{thm}{name = \RSthmtxt}}
  {}
\RS@ifundefined{lemref}
  {\def\RSlemtxt{lemma~}\newref{lem}{name = \RSlemtxt}}
  {}

\providecolor{lyxadded}{rgb}{0,0,1}
\providecolor{lyxdeleted}{rgb}{1,0,0}

\numberwithin{equation}{section}
\numberwithin{figure}{section}
\theoremstyle{plain}
\newtheorem{thm}{\protect\theoremname}[section]
  \theoremstyle{plain}
  \newtheorem{lem}[thm]{\protect\lemmaname}
  \theoremstyle{plain}
  \newtheorem{cor}[thm]{\protect\corollaryname}
  \theoremstyle{plain}
  \newtheorem{prop}[thm]{\protect\propositionname}

\newtheorem{remark}[thm]{Remark}

\theoremstyle{remark}

\title{\large{\textbf{A LOWER BOUND FOR DISCONNECTION \linebreak BY RANDOM INTERLACEMENTS}}}

\date{}

\usepackage[english]{babel}

\addtocounter{section}{-1}

\makeatother

\usepackage{babel}
  \providecommand{\corollaryname}{Corollary}
  \providecommand{\lemmaname}{Lemma}
  \providecommand{\propositionname}{Proposition}
  
\providecommand{\theoremname}{Theorem}

\begin{document}
\maketitle \thispagestyle{empty}
\begin{center} \vspace{-0.7cm}  Xinyi Li and Alain-Sol Sznitman
\end{center} \vspace{1.0cm}
\begin{center}
\end{center}
\begin{abstract}
\centering
\begin{minipage}{0.9\textwidth}
We consider the vacant set of random interlacements on $\mathbb{Z}^d$, $d \geq 3$, in the percolative regime. Motivated by the large deviation principles recently obtained in \cite{LiSzni13},  we investigate the asymptotic behavior of the probability that a large body gets disconnected from infinity by the  random interlacements. We derive an asymptotic lower bound, which brings into play tilted interlacements,  and relates the problem to some of the large deviations of the occupation-time profile considered in \cite{LiSzni13}.
\end{minipage} \end{abstract}

\vfill
\noindent
{\small
-------------------------------- \\
Departement Mathematik, ETH Z\"urich, CH-8092 Z\"urich, Switzerland. \\
This research was supported in part by the grant ERC-2009-AdG 245728-RWPERCRI.
}

\newpage
\thispagestyle{empty} \mbox{}
\newpage  \pagestyle {plain}

\section{Introduction}

Random interlacements constitute a percolation model with long-range dependence, and the percolative properties of their vacant set play an important role in the investigation of several questions of disconnection or fragmentation created by random walks, see \cite{CernTeixWind11}, \cite{Szni09b}, \cite{TeixWind11}. Here, we consider random interlacements on $\mathbb{Z}^{d}$, $d\ge3$. It is by now well-known that as one increases the level $u$ of the interlacements, the percolative properties of the vacant set undergo a phase transition, and the model evolves from a percolative phase to a non-percolative phase, see \cite{Szni10a} and \cite{SidoSzni09}. In the present work, we are mainly interested in the percolative phase of the model, and we derive an asymptotic lower bound on the probability that a macroscopic body has no connection to infinity in the vacant set. Strikingly, this lower bound corresponds to certain large deviations of the occupation-time profile of random interlacements investigated in our previous work \cite{LiSzni13},
where we analyzed the exponential decay of the probability that a macroscopic body gets insulated by high values of the (regularized) occupation-time profile.

\medskip
We now describe the model and our results in a more precise fashion. We refer to Section 1 for precise definitions. We consider continuous-time random interlacements on $\mathbb{Z}^{d}$, $d\geq3$. We denote by $\mathbb{P}_{u}$ the canonical law of random interlacements at level $u>0$, and by $\mathcal{I}^{u}$ and $\mathcal{V}^{u} = \mathbb{Z}^{d}\backslash\mathcal{I}^{u}$ the corresponding interlacement set and vacant set. It is known that there is a critical value $u_{**}\in(0,\infty)$, which can be characterized as the infimum of the levels $u>0$ for which the probability that the vacant cluster at the origin reaches distance $N$ from the origin has a  stretched exponential decay in $N$, see \cite{SidoSzni10}. It is an important open question whether $u_{**}$ actually coincides with the critical level $u_{*}$ for the percolation of the vacant set (but it is a simple fact that $u_{*} \le u_{**}$).

\medskip
In this work, we are primarily interested in the percolative regime of the vacant set, but, specifically, we assume that $0 < u \le u_{**}$ (because our lower bound on disconnection actually provides information in this possibly wider range of levels).

\medskip
We consider a compact subset $K$ of $\mathbb{R}^{d}$, and its discrete blow-up:
\begin{equation}\label{eq:blowupdef}
K_{N}=\{x\in\mathbb{Z}^{d};~d_{\infty}(x,NK)\le1\},
\end{equation}
where $NK$ denotes the homothetic of ratio $N$ of the set $K$, and $d_{\infty}(z,NK)=\inf_{y\in NK}|z-y|_{\infty}$ stands for the sup-norm distance of $z$ to $NK$.
Of central interest for us is the event stating that $K_N$ is not connected to infinity in $\mathcal{V}^{u}$, which we denote by
\begin{equation}\label{eq:ANdef}
A_{N}=\{K_{N}\; \overset{\mathcal{V}^{u}}{\mbox{\Large $\nleftrightarrow$}}\; \infty\}.
\end{equation}

\medskip\n
The main result of this article is the following asymptotic lower bound.
\begin{thm}\label{thm:lowerbound}
For $ u\in(0,u_{**}]$ one has
\begin{equation}\label{eq:mainthm}
\liminf_{N\to\infty}\frac{1}{N^{d-2}}\log(\mathbb{P}_{u}[A_{N}])\ge-\frac{1}{d}(\sqrt{u_{**}}-\sqrt{u})^{2}\mathrm{cap}_{\mathbb{R}^{d}}(K),
\end{equation}
where $\mathrm{cap}_{\mathbb{R}^{d}}(K)$ stands for the Brownian capacity of $K$.
\end{thm}

In essence, the lower bound (\ref{eq:mainthm}) replicates the asymptotic behavior of the probability that the regularized occupation-time profile of random interlacements insulates $K$ by values exceeding $u_{**}$, see Theorems 6.2 and 6.4, as well as Remarks 6.5 2) and  6.5 5) of \cite{LiSzni13}. It is a remarkable feature that such large deviations of the occupation-time profile induce a ``thickening'' of the interlacement surrounding $K_N$, rather than a mere change of the clocks governing the time spent by the trajectories defining the interlacement. This thickening is potent enough to typically disconnect $K_N$ from infinity. We refer to Remark \ref{rem2.5} for more on this topic. It is of course an important question, whether there is a matching upper bound to (\ref{eq:mainthm}), when $K$ is a smooth compact, and whether the large deviations of the occupation-time profile capture the main mechanism through which $\mathcal{I}^{u}$ disconnects a macroscopic body from infinity.

\medskip
Incidentally, the tilted interlacements, which we heavily use in this work, come up as a kind of slowly space-modulated random interlacements. Possibly, they offer, in a discrete set-up, a microscopic model for the type of ``Swiss cheese'' picture advocated in \cite{VandBoltHoll01}, when studying the moderate deviations of the volume of the Wiener sausage (however the relevant modulating functions in \cite{VandBoltHoll01} and in the present work correspond to distinct variational problems and are different).

\medskip
One may also compare Theorem \ref{thm:lowerbound} to corresponding
results for supercritical Bernoulli percolation. Unlike what happens in the present set-up, disconnecting a large macroscopic body in the percolative phase (when $K$ is a smooth compact) would involve an exponential cost proportional to $N^{d-1}$, in the spirit of the study of the existence of a large finite cluster at the origin, see p.~216 of  \cite{Grim99}, or Theorem 2.5, p.~16 of \cite{Cerf00}.

\medskip
Further, it is interesting to note that when $u\to0$, the right-hand
side of (\ref{eq:mainthm}) has a finite limit. One may wonder about the relation of
this limit to what happens in our original problem when one
replaces $\mathcal{I}^{u}$ by a single random walk trajectory (starting
for instance at the origin), that is, when we consider the probability
that $K_{N}$ is disconnected from infinity by the trajectory of one single random walk starting at the origin.
We refer to Remark \ref{endremark} 2) for more on this question.

\medskip
We briefly comment on the proofs. The main strategy is to use a change
of probability and an entropy bound. We construct through fine-tuned Radon-Nikodym derivatives
new measures $\widetilde{\mathbb{P}}_{N}$ corresponding to ``tilted random
interlacements'', which have the crucial property that under $\widetilde{\mathbb{P}}_{N}$ the disconnection probability tends
to 1 as $N$ goes to infinity:
\begin{equation}\label{0.4}
\widetilde{\mathbb{P}}_{N}[A_{N}]\to1.
\end{equation}
Then, by a classical inequality (see (\ref{eq:Entropychange})), one has
a lower bound for the disconnection probability in terms of the relative entropy:
\begin{equation}\label{0.5}
\liminf_{N\to\infty}\frac{1}{N^{d-2}}\log(\mathbb{P}_{u}[A_{N}])\geq-\limsup_{N\to\infty}\frac{1}{N^{d-2}}H(\widetilde{\mathbb{P}}_{N}|\mathbb{P}_{u}).
\end{equation}

\n
We relate the relative entropy of $\widetilde{\mathbb{P}}_{N}$ with respect to $\mathbb{P}_{u}$,  to the Brownian capacity of $K$, and show in Propositions \ref{prop:entropycalc} and \ref{prop:limsup} that
\begin{equation}\label{0.6}
\widetilde{\lim}\frac{1}{N^{d-2}}H(\widetilde{\mathbb{P}}_{N}|\mathbb{P}_{u}) = -\frac{1}{d}(\sqrt{u_{**}}-\sqrt{u})^{2}\mathrm{cap}_{\mathbb{R}^{d}}(K)
\end{equation}
(where $\widetilde{\lim}$ refers to certain successive limiting procedures involving $N$ first, and then various auxiliary parameters entering the construction of $\widetilde{\mathbb{P}}_{N}$).

\medskip
The measure $\widetilde{\mathbb{P}}_{N}$ governing the tilted interlacements is constructed in Section 2. Intuitively, it forces a ``local level'' of interlacements corresponding to $u_{**}+\epsilon$,  in a ``fence'' surrounding $K_N$. This creates a strongly non-percolative region surrounding $K_N$ and leads to (\ref{0.4}). Of course, a substantial part of the work is to make sense of the above heuristics. This goes through a local comparison at a mesoscopic scale between the occupied set of tilted interlacements and standard interlacements at a level exceeding $u_{**}$.

\medskip
In particular, we show in Proposition \ref{prop:coupling} that for all mesoscopic boxes $B_1$, with size $N^{r_1}$ (with $r_1$ small) and center in $\Gamma^{N}$, a ``fence'' around $K_N$, one has a coupling $\bar{Q}$ between $\mathcal{I}_{1}$, distributed as $\mathcal{I}^{u_{**}+\epsilon/8}\cap B_{1}$, and $\widetilde{\mathcal{I}}$, distributed as the intersection of the titled interlacement set with $B_1$, so that
 \begin{equation}\label{0.7}
\bar{Q}[\widetilde{\mathcal{I}}\supset\mathcal{I}_{1}]\geq1-ce^{-c'N^{c''}}.
\end{equation}

\medskip
The proof of this key stochastic domination bound relies on two main ingredients. On the one hand, it involves a comparison of equilibrium measures, see Proposition \ref{lem:surfacedom}, which itself relies on a comparison of capacities on a slightly larger mesoscopic scale, see Proposition \ref{prop:Capcompare}. On the other hand, it involves a domination of $\mathcal{I}^{u_{**}+\epsilon/8}\cap B_{1}$ by the trace on $B_1$ of a suitable Poisson point process of excursions of the simple random walk starting on the boundary of $B_{1}$ up to their exit from a larger box $B_2$. For this last step we can rely on results of \cite{Beli13}.

\medskip
We will now explain how this article is organized. In Section 1 we introduce notation and make a brief review of results concerning continuous-time random walk, Green function, continuous-time random interlacements, as well as other useful facts and tools. Section 2 is devoted to the construction of the probability measure governing the tilted random interlacements. We also compute and obtain asymptotic estimates on the relative entropy, see Propositions \ref{prop:entropycalc} and \ref{prop:limsup}.  In Section 3 we derive
a comparison of capacities in Proposition \ref{prop:Capcompare}, and, subsequently, of equilibrium measures in Proposition \ref{prop:mesdom}. The latter proposition plays a crucial role in the
construction of the coupling in the next section. In Section 4 we prove (\ref{0.7}) in Proposition \ref{prop:coupling}, and the crucial statement (\ref{0.4}) in Theorem \ref{thm:tiltedone}. In the short Section 5 we assemble
the various pieces and prove the main theorem.

\medskip
Finally, we explain the convention we use concerning constants.
We denote by $c, c', \bar{c}, \widetilde{c}\ldots$ positive constants with
values changing from place to place, and by $c_{0}, c_{1},\ldots$
positive constants which are fixed and refer to the value as they
first appear. Throughout the article the constants depend on the dimension
$d$. Dependence on additional constants are stated explicitly in the notation.

\section{Some useful facts}

Throughout the article we assume $d\geq3$. In this section we
introduce further notation and useful facts,
in particular concerning continuous time random walk on $\mathbb{Z}^{d}$
and its potential theory. The Lemma \ref{lem:MEA2} concerns the occupation-times
of balls and will be used in Section 3. Moreover, we introduce another
continuous-time reversible Markov chain on $\mathbb{Z}^{d}$, which
will play a crucial role in the upcoming sections, and we state some useful
results regarding its potential theory. We also recall the definition and basic facts concerning continuous time random interlacements. We end this section by stating
some results about relative entropy and Poisson point processes.

\medskip
We start with some notation. We let $\mathbb{N}=\{0,1,\ldots\}$
stand for the set of natural numbers. We write $|\cdot|$ and $|\cdot|_{\infty}$
for the Euclidean and $l^{\infty}$-norms on $\mathbb{R}^{d}$. We
denote by $B(x,r)=\{y\in\mathbb{Z}^{d};\:|x-y|\leq r\}$ the closed Euclidean
ball of radius $r\geq0$ intersected with $\IZ^d$, and respectively
by $B_{\infty}(x,r)=\{y\in\mathbb{Z}^{d},\:|x-y|_{\infty}\leq r\}$
the closed $l^{\infty}$-ball of radius $r$ intersected with $\mathbb{Z}^{d}$.
When $U$ is a subset of $\mathbb{Z}^{d}$, we write $|U|$ for the
cardinality of $U$, and $U\subset\subset\mathbb{Z}^{d}$ means that $U$ is a finite subset of $\mathbb{Z}^{d}$. We denote by $\partial U$ (resp. $\partial_{i}U$)
the boundary (resp. internal boundary) of $U$, and by $\overline{U}$ its ``closure'':
\begin{equation}\label{eq:boundarydef}
\begin{split}
\partial U & =\{x\in U^{c};\:\exists y\in U,\ |x-y|=1\},
\\
\partial_{i}U & =\{x \in U; \; \exists y \in U^c, \; |x-y| = 1\}, \;\mbox{and}\; \bar{U}=U\cup\partial U \,.
\end{split}
\end{equation}

\medskip\n
When $U\subset\mathbb{\mathbb{R}}^{d}$ , and $\delta>0$ , we write
$U^{\delta}=\{z\in\mathbb{R}^{d};\: d(z,U)\leq\delta\}$ for the closed $\delta$-neighborhood of $U$, where $d(x,A)=\inf_{y\in A}|x-y|$ is the distance function. We define $d_{\infty}(x,A)$ in a similar fashion, with $|\cdot|_{\infty}$ in place of $|\cdot|$.
To distinguish balls in $\mathbb{R}^{d}$ from balls in $\mathbb{Z}^{d}$,
we write $B_{\mathbb{R}^{d}}(x,r)=\{z\in\mathbb{R}^{d}$; $|x-z|\leq r\}$
for the (closed) Euclidean ball of radius $r$ in $\mathbb{R}^{d}$.
We also introduce the $N$-discrete blow-up of $U$ as
\begin{equation}\label{eq:blowupdef-1}
U_{N}=\{x\in\mathbb{Z}^{d}; \; d_{\infty}(x,NU)\le1\},
\end{equation}
where $NU=\{Nz;\, z\in U\}$ denotes the homothetic of $U$.

\medskip
We will now collect some notation concerning connectivity properties. We call $\pi:\{1,\ldots n\}\to\mathbb{Z}^{d}$, with $n\geq1$, a nearest-neighbor path, when $|\pi(i)-\pi(i-1)|=1$,
for $1<i\leq n$. Given $K,L,U$ subsets of $\mathbb{Z}^{d}$, we say that
$K$ and $L$ are connected by $U$ and write $K\overset{U}{\leftrightarrow}L$, if there exists a finite nearest-neighbor path $\pi$ in $\mathbb{Z}^{d}$
such that $\pi(1)$ belongs to $K$ and $\pi(n)$ belongs to $L$, and for all $k$ in $\{1,\cdots,n\}$,
$\pi(k)$ belongs to $U$. Otherwise, we say that $K$ and $L$ are not connected
by $U$, and write $K\overset{U}{\nleftrightarrow}L$. Similarly, for $K,U\subset\mathbb{Z}^{d}$,
we say that $K$ is connected to infinity by $U$, if 
$K\overset{U}{\leftrightarrow}B(0,N)^{c}$ for all $N$, and write
$K\overset{U}{\leftrightarrow}\infty$. Otherwise, we say that $K$ is not connected to
infinity by $U$, and denote it by $K\overset{U}{\nleftrightarrow}\infty$.

\vspace{0.3cm}

We now turn to the definition of some path spaces, and of the continuous-time simple random
walk. We consider $\widehat{W}_{+}$ and $\widehat{W}$ the spaces
of infinite (resp. doubly-infinite) $(\mathbb{Z}^{d})\times(0,\infty)$-valued
sequences such that the first coordinate of the sequence forms an
infinite (resp. doubly-infinite) nearest-neighbor path in $\mathbb{Z}^{d}$,
spending finite time in any finite subset of $\mathbb{Z}^{d}$, and the
sequence of the second coordinate has an infinite sum (resp. infinite
``forward'' and ``backward'' sums). The second coordinate describes the duration at each step corresponding to the first coordinate. We denote by $\widehat{\mathcal{W}}_{+}$
and $\widehat{\mathcal{W}}$ the respective $\sigma$-algebras generated
by the coordinate maps. We denote by $P_{x}$ the law on $\widehat{W}_{+}$
under which $Z_{n}$, $n\geq0$, has the law of the simple random walk
on $\mathbb{Z}^{d}$, starting from $x$, and $\zeta_{n}$, $n\geq0$,
are i.i.d. exponential variables with parameter $1$, independent
from $Z_{n}$, $n\geq0$. We denote by $E_{x}$ the corresponding expectation.
Moreover, if $\alpha$ is a measure on $\mathbb{Z}^{d}$, we denote
by $P_{\alpha}$ and $E_{\alpha}$ the measure $\sum_{x\in\mathbb{Z}^{d}}\alpha(x)P_{x}$ (not necessarily a probability measure) and its corresponding ``expectation'' (i.e.~the integral with respect to the measure $P_\alpha$).

\medskip
We attach to $\widehat{w}\in\widehat{W}_{+}$ a continuous-time process $(X_t)_{t \ge 0}$, and call it the random walk on $\mathbb{Z}^{d}$ with constant jump
rate $1$ under $P_{x}$, through the following relations
\begin{equation}\label{1.3}
X_{t}(\widehat{w})=Z_{k}(\widehat{w}),\; \textrm{for $t \ge 0$, when }\sum_{i=0}^{k-1}\zeta_{i}\leq t<\sum_{i=0}^{k}\zeta_{i}
\end{equation}
(if $k=0$, the left sum term is understood as 0).
We also introduce the filtration
\begin{equation}\label{1.4}
\mathcal{F}_{t}=\sigma(X_{s},\: s\leq t), \; t \ge 0.
\end{equation}

\medskip\n
Given $U\subseteq \mathbb{Z}^{d}$, and $\widehat{w}\in\widehat{W}_{+}$,
we write $H_{U}(\widehat{w})=\inf\{t\geq0;\: X_{t}(\widehat{w})\in U\}$
and $T_{U}=\inf\{t\geq0;\: X_{t}(\widehat{w})\notin U\}$ for the
entrance time in $U$ and exit time from $U$. Moreover, we write
$\widetilde{H}_{U}=\inf\{s\geq\zeta_{1}; X_{s}\in U$\} for the hitting time of $U$.

\medskip
For $U\subset\mathbb{Z}^{d}$, we write $\Gamma(U)$ for the space
of all right-continuous, piecewise constant functions from $[0,\infty)$ to $U$,
with finitely many jumps on any compact interval. We will also denote by $(X_t)_{t\ge0}$ the canonical coordinate process on $\Gamma(U)$, and whenever an ambiguity arises, we will specify on which space we are working.

\medskip
We denote by $g(\cdot,\cdot)$ and $g_{U}(\cdot,\cdot)$ the Green
function of the walk, and the killed Green function of the walk upon leaving $U$,
\begin{equation}\label{1.5}
g(x,y)=E_{x}\Big[\int_{0}^{\infty}1_{\{X_{s}=y\}}ds\Big],\quad g_{U}(x,y)=E_{x}\Big[\int_{0}^{T_{U}}1_{\{X_{s}=y\}}ds\Big].
\end{equation}
It is known that $g$ is translation invariant. Moreover, both $g$ and $g_{U}$
are symmetric and finite, that is,
\begin{equation}\label{eq:Greensym}
g(x,y)=g(y,x),\quad g_{U}(x,y)=g_{U}(y,x)\textrm{ for all }x,y\in\mathbb{Z}^{d}.
\end{equation}

\medskip\n
When $x$ tends to infinity, one knows that (see, e.g. p.~153, Proposition
6.3.1 of \cite{LawlLimi10})
\begin{equation}\label{eq:greenasym}
g(0,x)=dG(x)+O(|x|^{1-d}),
\end{equation}
where for $y\in\mathbb{R}^{d}$
\begin{equation}\label{eq:contgreendef}
G(y)=c_{0}|y|^{2-d}
\end{equation}

\medskip\n
is the Green function with a pole at the origin, attached to Brownian motion, and
\begin{equation}\label{1.8}
c_{0}=\frac{\bar{c}_{0}}{d}=\frac{1}{2\pi^{d/2}}\;\Gamma \Big(\frac{d}{2}-1\Big).
\end{equation}

\medskip\n
We also have the following estimate on the killed Green function (see
p.~157, Proposition 6.3.5 of \cite{LawlLimi10}): for $x\in B(0,N)$,
\begin{equation}\label{eq:gUestimate}
\begin{split}
g_{B(0,N)}(0,x)  = &  \; g(0,x) - E_{x}\big[ g(0, X_{T_{B(0,N)}})\big] \\[1ex]
 = &\; \bar{c}_{0}(|x|^{2-d}-N^{2-d})+O(|x|^{1-d}).
\end{split}
\end{equation}

\medskip
We further recall the definitions of equilibrium measure and capacity, and refer to Section 2, Chapter 2 of \cite{Lawl91} for more details. Given $M\subset\subset\mathbb{Z}^{d}$,
and we write $e_{M}$ for the equilibrium measure of $M$:
\begin{equation}\label{eq:emdef}
e_{M}(x)=P_{x}[\widetilde{H}_{M}=\infty]1_{M}(x),\: x\in\mathbb{Z}^{d},
\end{equation}
and $\mathrm{cap}(M)$ for the capacity of $M$, which is the total
mass of $e_{M}$:
\begin{equation}\label{1.12}
\mathrm{cap}(M)=\sum_{x\in K}e_{M}(x).
\end{equation}

\medskip\n
There is also an equivalent definition of capacity through the Dirichlet
form:
\begin{equation}\label{1.13}
\mathrm{cap}(M)=\inf_{f}\mathcal{E}_{\mathbb{Z}^{d}}(f,f)
\end{equation}
where $f:\mathbb{Z}^{d}\to\mathbb{R}$ is finitely supported, $f\geq1$
on $M$, and
\begin{equation}\label{eq:Dirichletdef}
\mathcal{E}_{\mathbb{Z}^{d}}(f,f)=\frac{1}{2}\sum_{|x-y|=1}\frac{1}{2d}\,\big(f(y)-f(x)\big)^{2}
\end{equation}
is the discrete Dirichlet form for simple random walk.

\medskip\n
Moreover, the probability of entering $M$ can be expressed as
\begin{equation}\label{1.15}
P_{x}[H_{M}<\infty]=\sum_{y\in M}g(x,y)e_{M}(y),
\end{equation}
and in particular, when $x\in M$, we have
\begin{equation}\label{eq:eKg1}
\sum_{y\in M}g(x,y)e_{M}(y)=1.
\end{equation}

\medskip\n
We now introduce some notation for (killed) entrance measures.
Given $A \subseteq B$ subsets of $\mathbb{Z}^d$, with $A$ finite,
we define for $x \in \IZ^d$, $y \in A$,
\begin{equation} \label{eq:hdefsf}
h_{A,B}(x,y)=P_{x}(H_{A}<T_{B},\: X_{H_{A}}=y) \,.
\end{equation}
When $B=\mathbb{Z}^{d}$, we simply write $h_{A}(x,z)$.

\medskip\n
The equilibrium measure also satisfies the sweeping identity (for
instance, seen as a consequence of (1.46) in \cite{Szni10a}), namely,
for $M\subset M'\subset\subset\mathbb{Z}^{d}$, $y \in M$, using the notation from above (\ref{1.3}),
\begin{equation}\label{eq:sweeping}
P_{e_{M'}}[H_{M}<\infty,X_{H_{M}}=y]=\sum_{x\in\partial_{i}M'}e_{M'}(x)h_{M}(x,y)=e_{M}(y).
\end{equation}

\medskip
The next lemma will be useful in Section 3, see Proposition \ref{prop:Capcompare}.
It provides an asymptotic estimate on the expected time a random walk starting
at the boundary of a ball of large radius spends in this ball. We recall
the convention on constants stated at the end of the Introduction.

\begin{lem}\label{lem:MEA2}
\begin{equation}\label{eq:alpha}
\alpha(N)\overset{\mathrm{def}}{=}\sup_{x\in\partial_{i}B(0,N)}\Big|\frac{E_{x}\big[\int_{0}^{\infty}1_{B(0,N)}(X_{s})ds\big]}{c_{1}N^{2}}-1\Big|\textrm{ tends to 0 as }N\to\infty
\end{equation}
\end{lem}
\begin{proof}
For simplicity, we fix $x$ in this proof and write $B(0,N)=B$. We set
\begin{equation}\label{eq:endef}
\epsilon_{N}=N^{-1/2},\: r_{N}=\epsilon_{N}N.
\end{equation}
 We split $B$ into two parts: $B_{I}= B \cap \widetilde{B}$ and $B_{J}=B\backslash \widetilde{B}$, where $\widetilde{B}=B(x,r_{N})$.

\medskip
In $B_{I}$, we use a crude upper bound for $g(x,\cdot)$, derived
from (\ref{eq:greenasym}),
\begin{equation}
g(x,y)\leq\frac{c}{(\max\{|x-y|_{\infty},1\})^{d-2}}.
\end{equation}
As a result, we find that
\begin{equation}\label{eq:Id}
\sum_{y\in B_{I}}g(x,y)\leq\sum_{l=0}^{\lceil r_{N}\rceil}\sum_{y:|y-x|_{\infty}=l}\frac{c}{(\max\{l,1\})^{d-2}}\leq c'r_{N}^{2}.
\end{equation}

\medskip\n
Let $\bar{x}=\frac{N}{|x|}x$ denote the projection of $x$ onto the
Euclidean sphere of radius $N$ centered at 0. It is straightforward
to see that
\begin{equation}\label{eq:Ic}
\int_{B_{\mathbb{R}^{d}}(\bar{x},r_{N})}G(y-\bar{x})dy\leq cr_{N}^{2}.
\end{equation}

\medskip\n
By the asymptotic approximation of discrete Green function (see (\ref{eq:greenasym})
and (\ref{eq:contgreendef})), writing $\widehat{B}=B_{\mathbb{R}^{d}}(0,N)\backslash B_{\mathbb{R}^{d}}(\bar{x},r_N)$, we obtain with a Riemann sum approximation argument that
\begin{equation}\label{eq:IIdc}
\begin{split}
\Big|\sum_{y\in B_J}g(x,y)-d\int_{\widehat{B}}G(\bar{y}-\bar{x})d\bar{y}  \Big|\leq & \;\Big|\sum_{y\in B_{J}}g(x,y)-d\int_{\widehat{B}}G(\bar{y}-x)d\bar{y}\Big|
\\[1ex]
  + &\; d\, \Big|\int_{\widehat{B}}(G(\bar{y}-x)-G(\bar{y}-\bar{x}))d\bar{y}\Big|
 \\[1ex]
  \leq &\;\l cN.
\end{split}
\end{equation}

\n
Thanks to the scaling property and rotation invariance of Brownian
motion, writing
\begin{equation}\label{1.25}
c_{1}=d\int_{B_{\mathbb{R}^{d}}(0,1)}G(\bar{y}-\bar{z})d\bar{y}\textrm{, where }\bar{z}\in\mathbb{R}^{d}\textrm{\textrm{ with }}|z|=1\textrm{ is arbitrary}
\end{equation}
 ($c_{1}/d$ is the expected time spent by Brownian motion in a ball
of radius 1 when starting from its boundary), and putting (\ref{eq:Id}),
(\ref{eq:Ic}) and (\ref{eq:IIdc}) together, we see that
\begin{equation}\label{1.26}
\Big|E_{x}\Big[\int_{0}^{\infty}1_{B(0,N)}(X_{s})ds\Big] - c_1 N^{2}\Big|\leq c r_{N}^{2}+c'N.
\end{equation}

\n
By the definition of $r_{N}$ in (\ref{eq:endef}), we obtain
(\ref{eq:alpha}) as desired.
\end{proof}

We now introduce a positive martingale, which plays an important  role in the definition of the tilted interlacements in the next section. We will show in the lemma
below that this martingale is uniformly integrable, and we will use its limiting
value as a probability density.

\medskip
Given a real-valued function $h$ on $\mathbb{Z}^{d}$, we denote its discrete Laplacian
by
\begin{equation}\label{1.27}
\Delta_{dis}h(x)=\frac{1}{2d}\sum_{|e|=1}h(x+e)-h(x).
\end{equation}

\n
We consider a positive function $f$ on $\mathbb{Z}^{d}$, which
is equal to 1 outside a finite set, and we write
\begin{equation}\label{1.28}
V=-\frac{\Delta_{dis}f}{f}.
\end{equation}
We also introduce the stochastic process
\begin{equation}\label{1.29}
M_{t}=\frac{f(X_{t})}{f(X_{0})}\exp\Big\{\int_{0}^{t}V(X_{s})ds\Big\}, t \ge 0,
\end{equation}
and define for all $x\in\mathbb{Z}^{d},\: T>0$ the positive measure
$\widetilde{P}_{x,T}$ (on $\widehat{W}_{+}$ with density $M_T$ with respect to $P_x$):
\begin{equation}\label{eq:tiltedP}
\widetilde{P}_{x,T}=M_{T}P_{x}.
\end{equation}

\n
The next lemma plays an important role in the construction of the tilted interlacements.
\begin{lem}\label{lem:martingale}
For all $x\in\mathbb{Z}^{d}$,
\begin{equation}\label{eq:mtmart}
(M_{t})_{t\geq0}\textrm{ is an }(\mathcal{F}_{t})\textrm{-martingale under $P_x$},
\end{equation}
and
\begin{equation}\label{eq:mtunifint}
(M_{t})_{t\geq0}\textrm{ is uniformly integrable under }P_{x}.
\end{equation}
Moreover,
\begin{equation}\label{eq:Minfty}
1=E_{x}[M_{\infty}]=\frac{1}{f(x)}E_{x}[e^{\int_{0}^{\infty}V(X_{s})ds}].
\end{equation}
\end{lem}
\begin{proof}
The first claim (\ref{eq:mtmart}) is classical. It follows for instance from Lemma 3.2, p.~174 in Chapter 4 of \cite{EthiKurt86}. Note that $E_{x}[M_{0}]=1$, so $\widetilde{P}_{x,T}$ is a probability measure for each $T$. Using the Markov property of $X$ under $P_x$ and (\ref{eq:mtmart}), it readily follows that $(X_t)_{0\leq t\leq T}$ under $\widetilde{P}_{x,T}$ is a Markov chain. By Theorem 2.5, p.~61 of \cite{DemuCaste00}, its semi-group (acting on the Banach space of functions on $\mathbb{Z}^{d}$ tending to zero at infinity) has a generator given by the bounded operator:
\begin{equation}\label{generatordef}
\begin{split}
\widetilde{L}h &=\frac{1}{f}\Delta_{dis}(fh)-\frac{\Delta_{dis}f}{f}h, \; \mbox{so that}
\\[1ex]
\widetilde{L} h(x) & = \dis\frac{1}{2d} \; \dis\sum\limits_{|e| = 1} \; \dis\frac{f(x+e)}{f(x)} \;\big(h(x + e) - h(x)\big)\,.
\end{split}
\end{equation}

\n
We introduce the law $\widetilde{Q}_{x}$ on $\Gamma(\mathbb{Z}^{d})$ of the jump process starting from $x$, corresponding to the generator $\widetilde{L}$ defined as in (\ref{generatordef}). Outside some finite set $f=1$, and by (\ref{generatordef}), outside the (discrete) closure of this finite set, this process jumps as a simple random walk. As a result, the canonical jump process attached to $\widetilde{Q}_x$ is transient. In addition, up to time $T$, it has the same law as $(X_t)_{0\leq t\leq T}$ under $\widetilde{P}_{x,T}$.

\medskip
Therefore, the claim (\ref{eq:mtunifint}) will follow once we show that
\begin{equation}\label{eq:delavalle}
\sup_{t\geq0}E_{x}[M_{t}\log M_{t}]=\sup_{T\geq t\geq0}\widetilde{E}_{x,T}[\log M_{t}]=\sup_{t\geq0} E^{\widetilde{Q}_x}[\log M_{t}]<\infty.
\end{equation}

\n
Now, setting $g=\log f$, we split $E^{\widetilde{Q}_{x}}[\log M_{t}]$
into two parts
\begin{align}
E^{\widetilde{Q}_{x}}[\log M_{t}]   = &\; E^{\widetilde{Q}_{x}}\Big[g(X_{t})-g(X_{0}) + \dis\int_{0}^{t}V(X_{s})ds\Big] \nonumber
\\
  = &\; E^{\widetilde{Q}_{x}}\Big[g(X_{t})-g(X_{0})-\dis\int_{0}^{t}\widetilde{L}g(X_{s})ds\Big] \label{eq:martsplit}
 \\
  + & \;E^{\widetilde{Q}_{x}}\Big[\dis\int_{0}^{t}(\widetilde{L}g+V)(X_{s})ds\Big]. \nonumber
\end{align}

\n
The first term after the second equality of (\ref{eq:martsplit}) is zero since $g(X_{t})-g(X_{0})-\int_{0}^{t}\widetilde{L}g(X_{s})ds$ is a martingale under $\widetilde{Q}_{x}$ (see Proposition 1.7, p.~162 of \cite{EthiKurt86}). As for the second term,
we write
\begin{equation}\label{1.38}
\psi=\widetilde{L}g+V.
\end{equation}
By (\ref{generatordef}) we see that
\begin{equation}\label{1.39}
\widetilde{L}g(x)=\frac{1}{2d}\sum_{|e|=1}\frac{f(x+e)}{f(x)}(g(x+e)-g(x)).
\end{equation}

\n
Hence, with a straightforward calculation and the fact that
\begin{equation}\label{1.40}
(1+u)\log(1+u)-u\geq0,\:\textrm{for }u>-1,
\end{equation}
we see that
\begin{equation}\label{eq:phipositive}
\psi(x)=\frac{1}{2d}\sum_{|e|=1}\Big(\frac{f(x+e)}{f(x)}\log\frac{f(x+e)}{f(x)}-\frac{f(x+e)-f(x)}{f(x)}\Big)\geq0,
\end{equation}
and that $\psi(x)$ is finitely supported.

\medskip
Therefore, due to the transience of the canonical process under $\widetilde{Q}_x$,
\begin{equation}\label{1.42}
\sup_{t\geq0}E^{\widetilde{Q}_{x}}\Big[\int_{0}^{t}\psi(X_{s})ds\Big]\overset{(\ref{eq:phipositive})}{\leq}E^{\widetilde{Q}_{x}}\Big[\int_{0}^{\infty}\psi(X_{s})ds\Big]<\infty,
\end{equation}
whence (\ref{eq:delavalle}).

\medskip
The last claim (\ref{eq:Minfty}) follows by uniform integrability. Indeed, the martingale converges
$P_{x}$-a.s. and in $L^{1}(P_{x})$ towards
\begin{equation}\label{1.43}
M_{\infty}=\frac{1}{f(X_{0})}\exp\Big\{\int_{0}^{\infty}V(X_{s})ds\Big\},
\end{equation}
so we have,
\begin{equation}\label{1.44}
E_{x}[M_{\infty}]= E_{x}[M_{0}]= 1.
\end{equation}
\end{proof}
We thus define for all $x$ in $\mathbb{Z}^{d}$ the positive measure on $\widehat{W}_{+}$:
\begin{equation}\label{1.45}
\widetilde{P}_{x}\overset{\mathrm{def}}{=}M_{\infty}P_{x}=\frac{1}{f(x)}\exp\Big\{\int_{0}^{\infty}V(X_{s})ds\Big\}P_{x}.
\end{equation}
The following corollary is a consequence of Lemma \ref{lem:martingale} and its proof.
\begin{cor}\label{cor1.3}
For all $x$ in $\mathbb{Z}^{d}$,
\begin{equation}\label{eq:tiltedPdef}
\widetilde{P}_{x}\textrm{ is a probability measure.}
\end{equation}
Moreover, $(X_{t})_{t\geq0}$ under $\widetilde{P}_{x}$, $x\in\mathbb{Z}^{d}$,
is a reversible Markov chain on $\mathbb{Z}^{d}$ with reversible
measure
\begin{equation}\label{eq:lambdadef}
\widetilde{\lambda}(x)=f^{2}(x),\: x\in\mathbb{Z}^{d},
\end{equation}
and its semi-group in $L^{2}(\widetilde{\lambda})$ has the bounded generator
\begin{equation}\label{eq:generatordef}
\widetilde{L}h(x)=(\frac{1}{f}\,\Delta_{dis}(fh)+ Vh)(x)=\frac{1}{2d}\sum_{|e|=1}\frac{f(x+e)}{f(x)}(h(x+e)-h(x)),
\end{equation}
for all $h$ in $L^2(\widetilde{\lambda})$ and $x$ in $\mathbb{Z}^{d}$. (Note that $X$ has variable jump rate under $\widetilde{P}_x$, unless $f$ is constant.)
\end{cor}
Similar to the results in potential theory for the continuous-time
simple random walk earlier in this section, we can also define for
$(X_{t})_{t\geq0}$ under $\{\widetilde{P}_{x}\}_{x\in\mathbb{Z}^{d}}$
the corresponding notions such as (killed) Green function, equilibrium
measure, and capacity. We also refer to Section 2.1 and 2.2 of Chapter
2 and Section 4.2 of Chapter 4 of \cite{FukuOshiTake11} for more details.
We denote the corresponding objects with a tilde, and refer to them
as tilted objects.

\medskip
Specifically, we write $\widetilde{g}$ and $\widetilde{g}_{U}$ for the tilted Green function
and killed Green function (outside $U \subseteq \IZ^d$):
\begin{equation}\label{eq:tiltedgdef}
\widetilde{g}(x,y)=\frac{1}{\widetilde{\lambda}(y)}\widetilde{E}_{x}\Big[\int_{0}^{\infty}1_{\{X_{s}=y\}}ds\Big],\quad\widetilde{g}_{U}(x,y)=\frac{1}{\widetilde{\lambda}(y)}\widetilde{E}_{x}\Big[\int_{0}^{T_{U}}1_{\{X_{s}=y\}}ds\Big].
\end{equation}
One knows that $\widetilde{g}$ and $\widetilde{g}_{U}$ are symmetric and
finite. Given $M\subset\subset\mathbb{Z}^{d}$, the tilted equilibrium
measure and tilted capacity of $M$ are defined as:
\begin{equation}\label{eq:tiltedemdef}
\widetilde{e}_{M}(x)=\widetilde{P}_{x}[\widetilde{H}_{M}=\infty]1_{M}(x)f(x) \Big(\dis\frac{1}{2d} \; \dis\sum\limits_{|e| = 1} f(x + e)\Big), \;\; \mbox{for} \; x\in\mathbb{Z}^{d}
\end{equation}
(the expression after the indicator function of $M$ is a reversibility measure of the discrete skeleton of the continuous-time chain, which can be viewed as a random walk among the conductances $\frac{1}{2d} f(x) f(y)$, for $x,y$ neighbors in $\IZ^d$, and $\widetilde{g}(\cdot,\cdot)$ is also the corresponding Green density of this discrete-time walk). Then (see (2.2.13), p.~79 of \cite{FukuOshiTake11})
\begin{equation}\label{1.50}
\mathrm{\widetilde{c}ap}(M)=\sum_{x\in M}\widetilde{e}_{M}(x).
\end{equation}

\n
Moreover, the following identities, analogues of (\ref{eq:eKg1})
and (\ref{eq:sweeping}), are valid:
\begin{equation}\label{eq:eKg1tilted}
\sum_{y\in M}\widetilde{g}(x,y)\widetilde{e}_{M}(y)=1, \;\; \mbox{for all} \; x\in M,
\end{equation}
and for $M\subset M'\subset\subset\mathbb{Z}^{d}$,
\begin{equation}\label{eq:sweeping-1}
\widetilde{P}_{\widetilde{e}_{M'}}[H_{M}<\infty,X_{H_{M}}=y]=\sum_{x\in M'}\widetilde{e}_{M'}(x)\widetilde{h}_{M}(x,y)=\widetilde{e}_{M}(y) \;\mbox{for all $y \in M$},
\end{equation}
where for $A \subseteq B \subseteq \IZ^d$, $x \in \IZ^d$, $y \in A$,
\begin{equation}\label{eq:tiltedhdef}
\widetilde{h}_{A}(x,y)=\widetilde{P}_{x}[H_{A}<\infty,X_{H_{A}}=y]\quad\widetilde{h}_{A,B}(x,y)=\widetilde{P}_{x}[H_{A}<T_{B},X_{H_{A}}=y]
\end{equation}
are the respective tilted entrance measure in $A$ and tilted entrance measure in $A$  relative to $B$, when starting at $x$.

\medskip
We now turn to continuous-time random interlacements. We refer to \cite{Szni12b} for more details. We define $\widehat{W}^{*}=\widehat{W}/\sim$,
where $\widehat{w}\sim\widehat{w}'$ is defined as $\widehat{w}(\cdot)=\widehat{w}'(\cdot + k)$ for some $k \in \IZ$, for $\widehat{w},\widehat{w}'\in\widehat{W}$.
We also define the canonical map as $\pi^{*}:\widehat{W}\to\widehat{W}^{*}$.
We write $\widehat{W}{}_{M}^{*}$ for the subset of $\widehat{W}^{*}$ of trajectories modulo time-shift that intersect $M\subset\subset\mathbb{Z}^{d}$. For $\widehat{w}^{*}\in\widehat{W}^{*}_{M}$, we write $\widehat{w}^{*}_{M,+}$ for the unique element of $\widehat{W}^{+}$, which follows $\widehat{w}^{*}$ step by step from the first time it enters $M$.

\medskip
The continuous-time random interlacement can be seen as a Poisson
point process on the space $\widehat{W}^{*}$, with intensity measure
$u\,\widehat{\nu}$, where $u>0$ and $\widehat{\nu}$ is a $\sigma$-finite
measure on $\widehat{W}$ such that its restriction to $\widehat{W}{}_{M}^{*}$
(denoted by $\widehat{\nu}_{M}$), is equal to $\pi^{*}\circ\widehat{Q}{}_{M}$,
where $\widehat{Q}{}_{M}$ is a finite measure on $\widehat{W}$ such that (see (1.7) in \cite{Szni12b})
if $(X_{t})_{t\in\mathbb{R}}$, is the continuous-time process attached
to $\widehat{w}\in\widehat{W}$, then
\begin{equation}\label{eq:rimespty0}
\widehat{Q}_{M}[X_{0}=x]=e_{M}(x),
\end{equation}
and when $e_{M}(x)>0$,
\begin{equation}\label{eq:rimespty}
\begin{array}{l}
\mbox{under $\widehat{Q}_{M}$ conditioned on $X_{0}=x,\:(X_{t})_{t\geq0}$ and the right-continuous}
\\
\mbox{regularization of $(X_{-t})_{t > 0}$ are independent and have same respective}
\\
\mbox{distribution as $(X_{t})_{t\geq0}$ under $P_{x}$ and $X$ after its first jump under}
\\
\mbox{$P_{x}[\cdot|\widetilde{H}_{M}=\infty]$}.
\end{array}
\end{equation}

\n
We define the space $\Omega$ of point measures on $\widehat{W}^{*}$ as
\begin{equation}\label{eq:1.57}
\Omega=\{\widehat{\omega}=\sum_{i\geq0}\delta_{\widehat{w}_{i}^{*}};\widehat{w}_{i}^{*}\in\widehat{W}^{*}\textrm{ for all }i\geq0,\:\widehat{\omega}(\widehat{W}{}_{M}^{*})<\infty\textrm{ for all }M\subset\subset\mathbb{Z}^{d}\}.
\end{equation}

\n
If $F:\widehat{W}^{*}\to\mathbb{R}$ and $\widehat{\omega}=\sum_{i}\delta_{\widehat{w}_{i}^{*}}$,
we write $<\widehat{\omega},F>=\sum_{i}F(\widehat{w}_{i}^{*})$ for
the integral of $F$ with respect to $\widehat{\omega}$. Given $M\subset\subset\mathbb{Z}^d$ and $\widehat{\omega} = \sum_{i\geq 0}\delta_{\widehat{w}^{*}_{i}}$ in $\Omega$, we let $\mu_M(\widehat{\omega})$ stand for the point measure on $\widehat{W}^{+}$, $\mu_M(\widehat{\omega})=\sum_{i\geq 0} 1_{\widehat{w}^{*}_{i}\in\widehat{W}^{*}_{M}}\delta_{(\widehat{w}^{*}_{i})_{M,+}}$, which collects the cloud of onward trajectories after the first entrance in $M$ (see below (\ref{eq:tiltedhdef}) for notation).

\medskip
We write $\mathbb{P}_{u}$ for the probability measure governing random
interlacements at level $u$, that is the canonical law on $\Omega$
of the Poisson point process on $\widehat{W}^{*}$ with intensity
measure $u\,\widehat{\nu}$. We write $\mathbb{E}_{u}$ for its
expectation. Given $\widehat{\omega}=\sum_{i}\delta_{\widehat{w}_{i}^{*}}$,
we define the interlacement set and vacant set at level $u$ respectively
as the random subsets of $\mathbb{Z}^{d}$:
\begin{equation}\label{eq:Iudef}
\mathcal{I}^{u}(\widehat{\omega})=\{\cup_{i}\mathrm{Range}(\widehat{w}_{i}^{*})\}
\end{equation}
where for $\widehat{w}^{*}$ in $\widehat{W}^{*}$, $\mathrm{Range}(\widehat{w}^{*})$ stands for the set of points in $\mathbb{Z}^{d}$ visited by any $\widehat{w}$ in
$\widehat{W}$ with $\pi^{*}(\widehat{w})=\widehat{w}^{*}$, and
\begin{equation}\label{eq:vudef}
\mathcal{V}^{u}=\mathbb{Z}^{d}\backslash(\mathcal{I}^{u}(\widehat{\omega})).
\end{equation}
The above random sets have the same law as $\mathcal{I}^{u}$ or $\mathcal{V}^{u}$ in \cite{Szni10a}.

\medskip
The connectivity function of the vacant set of random interlacements is known to have a stretched-exponential
decay when the level exceeds a certain critical value
(see Theorem 4.1 of \cite{Szni12a}, or Theorem 0.1 of \cite{SidoSzni10}, and Theorem 3.1 of \cite{PopoTeix13} for recent developments). Namely, there exists a $u_{**} \in (0,\infty)$, which, for our purpose in this article, can be characterized as the smallest positive number such that for all $u>u_{**}$,
\begin{equation}\label{eq:supercrit}
\mathbb{P}_{u}[0\overset{\mathcal{V}^{u}}{\leftrightarrow}\partial B_{\infty}(0,N)]\leq c_{2}(u)e^{-c_{3}(u)N^{c_{4}(u)}}\textrm{, for all }N\geq0.
\end{equation}

\medskip\n
(actually, by Theorem 3.1 of \cite{PopoTeix13}, one can choose $c_4 = 1$, when $d \ge 4$, and $c_4 = \frac{1}{2}$ or any other value in $(0,1)$, when $d=3$).

\medskip
We also wish to recall a classical result on relative entropy which will be helpful in Section 2. For $\widetilde{\mathbb{P}}$ absolutely continuous with respect to $\mathbb{P}$, the relative entropy of $\widetilde{\mathbb{P}}$ with respect to $\mathbb{P}$ is defined as
\begin{equation}\label{eq:relentrodef}
H(\widetilde{\mathbb{P}}|\mathbb{P})=\widetilde{\mathbb{E}}\Big[\log\frac{d\mathbb{\widetilde{P}}}{d\mathbb{P}}\Big]=\mathbb{E}\Big[\frac{d\mathbb{\widetilde{P}}}{d\mathbb{P}}\log\frac{d\mathbb{\widetilde{P}}}{d\mathbb{P}}\Big]\in[0,\infty].
\end{equation}

\smallskip\n
For an event $A$ with positive $\widetilde{\mathbb{P}}$-probability,
we have the following inequality (see p.~76 of \cite{DeusStro89}):
\begin{equation}\label{eq:Entropychange}
\mathbb{P}[A]\ge\widetilde{\mathbb{P}}[A]e^{-\frac{1}{\widetilde{\mathbb{P}}[A]}(H(\widetilde{\mathbb{P}}|\mathbb{P})+\frac{1}{e})}.
\end{equation}

\smallskip\n
We end this section by recalling one property of the Poisson point
process on general spaces. It rephrases Lemma 1.4 of \cite{LiSzni13}. Let
$\mu$ be a Poisson point process on $E$ with fi{}nite intensity
measure $\eta$ (i.e. $\eta(E)<\infty$), and let $\Phi:E\to\mathbb{R}$
be a measurable function. Then, one has
\begin{equation}\label{eq:Laplacian}
\textrm{ }E[e^{<\omega,\Phi>}]=e^{\int_{E}e^{\Phi}-1d\mu}
\end{equation}

\n
(this is an identity in $(0,+ \infty]$).

\bigskip
\section{The tilted interlacements}

In this section, we define a new probability measure $\widetilde{\mathbb{P}}_{N}$
on $\widehat{W}^{*}$, which is absolutely continuous with respect
to $\mathbb{P}_{u}$, see Proposition \ref{prop:tiltedPdef}. It governs a Poisson point process on $\widehat{W}^{*}$, which corresponds to the ``tilted random interlacements''. Intuitively, these tilted interlacements describe a kind of slowly space-modulated random interlacements. The motivation for the exponential tilt entering the definition of $\widetilde{\mathbb{P}}_{N}$ actually stems from the analysis of certain large deviations of the occupation-time profile of random interlacements considered in \cite{LiSzni13}, see Remark \ref{rem2.5} below. In Proposition \ref{prop:tiltedPdef} we compute the relative entropy of $\widetilde{\mathbb{P}}_{N}$ with respect to $\mathbb{P}_{u}$, and we then relate this result to the capacity of $K$ after a suitable limiting procedure, see Proposition \ref{prop:limsup}.

\medskip
We begin with the construction of the new measure $\widetilde{\mathbb{P}}_{N}$, which will correspond to an exponential tilt of $\mathbb{P}_{u}$, see (\ref{eq:Pproba}).

\medskip
We recall that $K$ is a compact subset of $\mathbb{R}^{d}$ as above
(\ref{eq:blowupdef}). We consider $\delta,\epsilon$ in $(0,1)$, and let $U$ and $\widetilde{U}$ be the open Euclidean balls centered at $0$ with respective radii $r_{U}$ and $r_{\widetilde{U}}$, where $r_U> 0$ and $r_{\widetilde{U}} = r_U + 4 $. We assume that $r_U$ is sufficiently large such that $K^{2\delta}\subset U\subset\widetilde{U}\subset\mathbb{R}^{d}$ (recall that $K^{2\delta}$ stands for the closed $2\delta$-neighborhood of $K$, see below (\ref{eq:boundarydef})).  By the end of this section we will eventually let $r_{U},r_{\widetilde{U}}$ tend to infinity and then let $\delta$ tend to 0. We denote by $W_z$ the Wiener measure starting from $z$ and by $H_F$, for $F$ a closed subset of $\mathbb{R}^d$, the entrance time of the canonical Brownian motion in $F$. We write
\begin{equation}\label{eq:hdef}
h(z)=W_{z}[H_{K^{2\delta}}<T_{U}],\ z\in\mathbb{R}^{d},
\end{equation}

\medskip\n
for the equilibrium potential of $K^{2\delta}$ relative to $U$. For
$\eta\in(0,\delta)$ and $\phi^{\eta}$ a non-negative smooth function
supported in $B_{\mathbb{R}^{d}}(0,\eta)$ such that $\int \phi^{\eta}(z) dz=1$,
we write
\begin{equation}\label{2.2}
h^{\eta}=h*\phi^{\eta}
\end{equation}
for the convolution of $h$ and $\phi^{\eta}$.

\medskip
We then define the restriction to $\mathbb{Z}^{d}$ of the blow-up of $h$ as
\begin{equation}\label{2.3}
h_{N}(x)=h^{\eta}\Big(\frac{x}{N}\Big),\; \mbox{for} \; x\in\mathbb{Z}^{d}.
\end{equation}
We now specify our choice of $f$ in (\ref{1.29}) as
\begin{align}
f(x) & =  \Big(\sqrt{\frac{u_{**}+\epsilon}{u}}-1\Big)\,h_{N}(x)+1,\label{eq:fdef}
\intertext{and recall that}
V & =  -\frac{\Delta_{dis}f}{f}.  \nonumber   
\end{align}

\medskip\n
$f$ and $V$ tacitly depend upon $\epsilon,\delta,\eta,N$.
We drop this dependence from the notation for the sake of simplicity. We denote by $\widetilde{U}_N$ the discrete blow-up of $\widetilde{U}$ (as in (\ref{eq:blowupdef}) or (\ref{eq:blowupdef-1})). We also note that
\begin{equation}\label{eq:fconstant}
\begin{array}{l}
\mbox{$f=1$ on ($\mathbb{Z}^d\backslash\widetilde{U}_N) \cup \overline{\partial_i \widetilde{U}_N}$, and for large $N$, $f=\sqrt{\frac{u_{**}+\epsilon}{u}}$ on $K^{\delta}_N$.}
\end{array}
\end{equation}
 From now on, we will denote by $\widetilde{P}_{x}$ the probability measure
defined in (\ref{1.45}), with $f$ as in (\ref{eq:fdef}).

\medskip
We define a function $F$ on $\widehat{W}^{*}$ through
\begin{equation}\label{2.5}
F(\widehat{w}^{*})=
\left\{ \begin{array}{l}
\dis\int_{0}^{\infty} \!\! V(X_{s})(\widehat{w}_{\widetilde{U}_N})ds, \; \mbox{for $\widehat{w}^{*}\in\widehat{W}_{\widetilde{U}_{N}}^{*}$, with $\pi^{*}(\widehat{w})=\widehat{w}^{*}$, and}
\\[2ex]
\mbox{$\widehat{w}_{\widetilde{U}_N}$ the time-shift of $\widehat{w}$ at its first entrance in $\widetilde{U}_N$,}
\\[3ex]
0,    \qquad  \textrm{otherwise.}
\end{array}\right.
\end{equation}

\medskip\n
We refer to (\ref{eq:1.57}) for the definition of $\Omega$.
\begin{prop}\label{prop:tiltedPdef}
\begin{equation}\label{eq:Pproba}
\mathbb{\widetilde{P}}_{N}=e^{<\widehat{\omega},F>}\mathbb{P}_{u}\textrm{ defines a probability measure on $\Omega$}.
\end{equation}
Moreover, under $\widetilde{\mathbb{P}}_{N}$
\begin{equation}\label{2.8}
\begin{array}{l}
\mbox{the canonical point measure $\widehat{\omega}$ is a Poisson point process on $\widehat{W}^{*}$ with }\\
\mbox{intensity measure $u\widetilde{\nu} $, where $\widetilde{\nu}=e^{F}\widehat{\nu}$},
\end{array}
\end{equation}
and for $M\subset\subset\mathbb{Z}^d$ (see below (\ref{eq:1.57}) for notation),
\begin{equation}\label{eq:tiltedentrancemes}
\begin{array}{l}
\mbox{$\mu_M$ is a Poisson point process on $\widehat{W}^+$ with intensity measure $u\widetilde{P}_{\widetilde{e}_M}$.}\\
\end{array}
\end{equation}
\end{prop}
\begin{proof}
We begin with the proof of (\ref{eq:Pproba}). By the first equality of (\ref{eq:fconstant}) and using (\ref{eq:Minfty}) of Lemma \ref{lem:martingale}, we see that for all $x\in\partial_{i}\widetilde{U}_{N}$,
\begin{equation}\label{2.10}
E_{x}[e^{\int_{0}^{\infty}V(X_{s})ds}]=1.
\end{equation}

\n
Since $F$ vanishes outside $\widehat{W}_{\widetilde{U}_{N}}^{*}$ , it follows that
\begin{equation}
\int_{\widehat{W}^{*}}(e^F-1)d\widehat{\nu} = \int_{\widehat{W}^*_{\widetilde{U}_N}} (e^F - 1) d \widehat{\nu} \overset{(\ref{eq:rimespty})}{=}E_{e_{\widetilde{U}_{N}}}[e^{\int_{0}^{\infty}V(X_{s})ds}-1] \stackrel{(\ref{2.10})}{=} 0,\label{eq:probverification}
\end{equation}
and by (\ref{eq:Laplacian}),
\begin{equation}
\mathbb{E}_u[e^{<\widehat{\omega},F>}]=1,
\end{equation}
whence (\ref{eq:Pproba}). We now turn to the proof of (\ref{2.8}).

\medskip
Writing $\widetilde{\mathbb{E}}_{N}$ as the expectation under $\widetilde{\mathbb{P}}_{N}$, taking $G$ a non-negative, measurable function on $\widehat{W}^{*}$,
we have
\begin{equation}\label{2.12}
\begin{array}{lcl}
\widetilde{\mathbb{E}}_{N}[e^{-<\widehat{\omega},G>}] &\!\!\!\!\!\overset{(\ref{eq:Pproba})}{=} &\!\!\!\!\! \mathbb{E}_{u}[e^{<\widehat{\omega},F-G>}]
\\[1ex]
 &\!\!\!\!\! \overset{(\ref{eq:probverification})}{=} &\!\!\!\!\! \mathbb{E}_{u}[e^{<\widehat{\omega},F-G>}]e^{-u\int(e^{F}-1)d\widehat{\nu}}
\\[1ex]
 &\!\!\!\!\! \underset{{\rm on} \; \widehat{W}_{\widetilde{U}_N}}{\stackrel{(\ref{eq:Laplacian})}{=}} &\!\!\!\!\!e^{u\int(e^{F-G}-1)d\widehat{\nu}-u\int(e^{F}-1)d\widehat{\nu}}
\\[1ex]
 &\!\!\!\!\!= &\!\!\!\!\! e^{u\int(e^{-G}-1)e^{F}d\widehat{\nu}}.
\end{array}
\end{equation}
This identifies the Laplace transform of $\widehat{\omega}$ under
$\widetilde{\mathbb{P}}_{N}$ and (\ref{2.8}) follows by Proposition 36,
p.~130 of \cite{Resn87}.

\smallskip
There remains to prove (\ref{eq:tiltedentrancemes}). By (\ref{2.8}) and the definition of $\mu_M$ (below (\ref{eq:1.57})), we see that
$\mu_M$ is a Poisson point process on $\widehat{W}^+$  with intensity measure $u\gamma_M$, where $\gamma_M$ is the image of $1_{\widehat{W}{}_{M}^{*}}\widetilde{\nu}$ under the map $\widehat{w}^{*} \rightarrow \widehat{w}^{*}_{M,+}$ (see above (\ref{eq:rimespty0}) for notation). The claim (\ref{eq:tiltedentrancemes}) will thus follow once we show that
\begin{equation}\label{2.13}
\gamma_M = \widetilde{P}_{\widetilde{e}_M}.
\end{equation}
We introduce $\widetilde{M}=M\cup \widetilde{U}_N$. We observe that
\begin{equation}\label{eq:equmeseqiv}
\widetilde{e}_{\widetilde{M}}=e_{\widetilde{M}}.
\end{equation}
Indeed, this follows by (\ref{eq:emdef}) and (\ref{eq:tiltedemdef}), together with the first equality in (\ref{eq:fconstant}).
We also note that in (\ref{2.5}) the function $F$ does not change if we replace 
$\widetilde{U}_N$ in the definition  by $\widetilde M$, since $\widetilde{U}_N \subset \widetilde M$, and $V$ vanishes outside $\widetilde{U}_N$. Therefore, in order to prove (\ref{2.13}), it suffices to verify that for any bounded measurable function $g:\widehat{W}^+\to \mathbb{R}$, its integral with respect to $\gamma_M$ coincides with that with respect to $\widetilde{P}_{\widetilde{e}_M}$.
We begin with $\langle\gamma_M,g\rangle$. By the definition of $\gamma_M$:
\begin{equation}
\begin{array}{lcl}
<\gamma_M,g> &\!\!\! = & \!\! \dis\int_{\widehat{W}^*_{\widetilde{M}}}e^{F}1_{\{\widehat{w}^*\in\widehat{W}^*_M\}}g(\widehat{w}^*_{M,+})d\widehat\nu(\widehat{w}^*)
\\[1ex]
&\!\! \underset{(\ref{eq:rimespty})}{\overset{(\ref{eq:rimespty0})}{=}} & \!\!\!E_{e_{\widetilde{M}}}[e^{\int_0^{\infty}V(X_s)ds}g(\widehat{w}_M)1_{\{H_M<\infty\}}],\\
\end{array}
\end{equation}
where for $\widehat{w}\in\widehat{W}^+$, we let $\widehat{w}_M\in\widehat{W}^+$ stand for the time-shift of $\widehat{w}$ starting at its first entrance in $M$.
We then apply the strong Markov property at $H_M$, and decompose according to where the walks enter $M$,
\begin{equation}
\begin{array}{lcl}\label{eq:gamma}
<\gamma_M,g> &\!\!\!\!\!\!\overset{\textrm{Markov}}{=} &\!\!\!\!
   E_{e_{\widetilde{M}}}[e^{\int_0^{H_M}V(X_s)ds}1_{\{H_M<\infty\}} E_{X_{H_M}}[e^{\int_0^{\infty}V(X_s)ds}g]]   
\\[1ex]
& \!\!\!\!\!\!= &\!\!\!\! E_{e_{\widetilde{M}}}[f(X_{H_M})e^{\int_0^{H_M}V(X_s)ds}1_{\{H_M<\infty\}} E_{X_{H_M}}\Big[\frac{1}{f(X_0)}e^{\int_0^{\infty}V(X_s)ds}g]\Big] 
\\[1ex]
&\!\!\!\!\!\! = &\!\!\!\!\!\!\!\!  \dis\sum\limits_{y\in\partial_i M} E_{e_{\widetilde{M}}}[f(y)e^{\int_0^{H_M}V(X_s)ds}1_{\{H_M<\infty,X_{H_M}=y\}}] E_{y}\Big[\textstyle\frac{1}{f(y)}e^{\int_0^{\infty}V(X_s)ds}g\Big]   
\\[1ex]
&\!\!\!\!\!\! \underset{\rm Markov}{\stackrel{(\ref{1.45})}{=}} &\!\!\!\!\!  \dis\sum\limits_{y\in\partial_i M} \widetilde{P}_{e_{\widetilde{M}}}[H_M<\infty,X_{H_M}=y]\widetilde{E}_y[g].   
\end{array}
\end{equation}

\n
On the other hand, we can express $\widetilde{P}_{\widetilde{e}_M}$ in terms of the tilted entrance measure by the sweeping identity (see (\ref{eq:sweeping-1})) and incorporate the fact that the tilted equilibrium measure of $\widetilde{M}$ coincides with the standard equilibrium measure of $\widetilde{M}$:
\begin{equation}\label{eq:tiltedPeM}
\begin{array}{lcl}
\widetilde{E}_{\widetilde{e}_{M}}[g] & \overset{(\ref{eq:sweeping-1})}{=} & \sum_{y\in\partial_i M} \widetilde{P}_{\widetilde{e}_{\widetilde{M}}}[H_M<\infty,X_{H_M}=y]\widetilde{E}_y[g]\
\\
& \overset{(\ref{eq:equmeseqiv})}{=} & \sum_{y\in\partial_i M} \widetilde{P}_{e_{\widetilde{M}}}[H_M<\infty,X_{H_M}=y]\widetilde{E}_y[g].
\end{array}
\end{equation}
Comparing (\ref{eq:gamma}) and (\ref{eq:tiltedPeM}), we obtain (\ref{2.13}).
\end{proof}

We will call the canonical Poisson point process under $\widetilde{\mathbb{P}}_{N}$
the tilted random interlacements.

\begin{remark}\label{rem2.2} \rm
The tilted interlacements do retain an interlacement-like character because $\widetilde{\nu} = e^F  \widehat{\nu}$ is a measure on $\widehat{W}^{*}$, which has the following property. Its restriction to $\widehat{W}_{M}^{*}$, for $M\subset\subset\mathbb{Z}^{d}$, is equal to $\pi^{*}\circ \widetilde{Q}{}_{M}$, where
\begin{equation}
\widetilde{Q}_{M}[X_{0}=x]=\widetilde{e}_{M}(x),
\end{equation}
and when $\widetilde{e}_{M}(x)>0$,
\begin{equation}\label{eq:rimespty-1}
\begin{array}{l}
\mbox{under $\widetilde{Q}_{M}$ conditioned on $X_{0}=x,\:(X_{t})_{t\geq0}$ and the right-continuous}
\\
\mbox{regularization of $(X_{-t})_{t > 0}$ are independent and with same respective}
\\
\mbox{distribution as $(X_{t})_{t\geq0}$ under $\widetilde{P}_{x}$ and $X$ after its first jump under}
\\
\mbox{$\widetilde{P}_{x}[\cdot|\widetilde{H}_{M}=\infty]$}.
\end{array}
\end{equation}

\medskip\n
We do not need the above fact, but mention it because it states the property analogous to (\ref{eq:rimespty0}) and (\ref{eq:rimespty}) satisfied by $\tilde{\nu}$.\hfill $\square$
\end{remark}

We will now calculate the relative entropy of $\widetilde{\mathbb{P}}_{N}$
with regard to $\mathbb{P}_{u}$ and relate it to the Dirichlet form
of $h_{N}$ (see (\ref{eq:Dirichletdef}) for notation).
\begin{prop}
\label{prop:entropycalc}
\begin{equation}
H(\widetilde{\mathbb{P}}_{N}|\mathbb{P}_{u}) = (\sqrt{u_{**}+\epsilon}-\sqrt{u})^{2}\mathcal{E}_{\mathbb{Z}^{d}}(h_{N},h_{N}).\label{eq:hdisc}
\end{equation}
\end{prop}
\begin{proof}
By the definition of relative entropy (see (\ref{eq:relentrodef})),
\begin{equation}\label{2.15}
H(\widetilde{\mathbb{P}}_{N}|\mathbb{P}_{u})=\widetilde{\mathbb{E}}_N[\log\frac{d\widetilde{\mathbb{P}}_{N}}{d\mathbb{P}_{u}}]  \stackrel{(\ref{eq:Pproba})}{=} \widetilde{\mathbb{E}}_N[<\widehat{\omega},F>],
\end{equation}
and
\begin{equation}\label{2.16}
\begin{array}{lcl}
\widetilde{\mathbb{E}}_N[<\widehat{\omega},F>] &\!\!\!\!\! = &\!\!\!\!\! u<\widetilde{\nu},F>
\\[2ex]
 & \!\!\!\!\!  \underset{(\ref{eq:tiltedentrancemes})}{\stackrel{(\ref{2.5})}{=}}
 &\!\!\!\!\! u\widetilde{E}_{\widetilde{e}_{\widetilde{U}_{N}}}\Big[\dis\int_{0}^{\infty}V(X_{s})ds\Big]
 \\[3ex]
 &\!\!\!\!\! \overset{(\ref{eq:tiltedgdef})}{=} &\!\!\!\!\! u\sum_{x\in\widetilde{U}_{N},\: x'\in\mathbb{Z}^{d}}\widetilde{e}_{\widetilde{U}_{N}}(x)\widetilde{g}(x,x')V(x')\widetilde{\lambda}(x')
 \\[3ex]
 & \!\!\!\!\!\underset{\mathrm{supp}\: V\subseteq\widetilde{U}_{N}}{\overset{(\ref{eq:eKg1tilted})}{=}} &\!\!\!\!\! u\sum_{x'\in\mathbb{Z}^{d}}V(x')\widetilde{\lambda} (x')
 \\[3ex]
 &\!\!\!\!\! \underset{(\ref{1.28})}{\overset{(\ref{eq:lambdadef})}{=}} &\!\!\!\!\! -u\sum_{x\in\mathbb{Z}^{d}}f(x)\Delta_{dis}f(x).
\end{array}
\end{equation}

\n
We also have, by the definition of $f$ in (\ref{eq:fdef}), that
\begin{equation}\label{eq:disclapcalc}
-u\sum_{x \in\mathbb{Z}^{d}}f(x)\Delta_{dis}f(x)=u\sum_{x\in\mathbb{Z}^{d}}\Big(\sqrt{\frac{u_{**}+\epsilon}{u}}-1\Big)f(x)\Delta_{dis}h_{N}(x)
\end{equation}
and since $h_{N}$ is finitely supported, by the Green-Gauss theorem,
the left-hand side of (\ref{eq:disclapcalc}) equals
\begin{equation}\label{2.19}
\begin{array}{cl}
= &  \!\!\!u\Big(\sqrt{\frac{u_{**}+\epsilon}{u}}-1\Big)\frac{1}{2}\sum_{|x-x'|=1}\frac{1}{2d}(f(x')-f(x))(h_{N}(x')-h_{N}(x))
\\[2ex]
\overset{(\ref{eq:fdef})}{=} &\!\!\!  u\sum_{x'\in\mathbb{Z}^{d}}\Big(\sqrt{\frac{u_{**}+\epsilon}{u}}-1\Big)^{2}\mathcal{E}_{\mathbb{Z}^{d}}(h_{N},h_{N}),
\end{array}
\end{equation}
and (\ref{eq:hdisc}) follows.
\end{proof}
We will now successively let $N\to\infty$, $\eta\to0$, $r_{U}\to\infty$,
and $\delta\to0$. The capacity of $K$ will appear in the limit (in
the above sense) of the properly scaled Dirichlet form of $h_{N}$.
\begin{prop}\label{prop:limsup}
\begin{equation}\label{eq:limsup}
\lim_{\delta\to0}\lim_{r_{U}\to\infty}\lim_{\eta\to0}\lim_{N\to\infty}\frac{1}{N^{d-2}}\mathcal{E}_{\mathbb{Z}^{d}}(h_{N},h_{N})=\frac{1}{d}\mathrm{cap}_{\mathbb{R}^{d}}(K).
\end{equation}
\end{prop}
\begin{proof}
First, by the definition of $h_{N}$ and (\ref{eq:Dirichletdef}) we have
\begin{align}
\frac{1}{N^{d-2}}\mathcal{E}_{\mathbb{Z}^{d}}(h_{N},h_{N}) = &\; \frac{1}{N^{d-2}}\sum_{x\in\mathbb{Z}^{d}}\sum_{|e|=1}\frac{1}{4d}(h_{N}(x+e)-h_{N}(x))^{2}\nonumber
\\ \label{2.21}
\underset{(\ref{2.3})}{\stackrel{(\ref{eq:fconstant})}{=}} &\;\frac{1}{4dN^{d}}\sum_{x\in\widetilde{U}_{N}}\sum_{|e|=1}N^2 \Big(h^{\eta}\Big(\frac{x+e}{N}\Big)-h^{\eta}\Big(\frac{x}{N}\Big)\Big)^{2}.
\end{align}

\n
Then, we take the limit of both sides. By the smoothness of $h^\eta$ and a Riemann sum argument we have:
\begin{equation}\label{2.22}
\lim_{N\to\infty}\frac{1}{N^{d-2}}\mathcal{E}_{\mathbb{Z}^{d}}(h_{N},h_{N})=\frac{1}{2d}\int|\nabla h^{\eta}(y)|^{2}dy=\frac{1}{d}\mathcal{E}_{\mathbb{R}^{d}}(h^{\eta},h^{\eta}),
\end{equation}

\n
where $\mathcal{E}_{\mathbb{R}^{d}}(\cdot,\cdot)$ denotes the usual Dirichlet form on $\IR^d$.

\medskip
Since $h$ in (\ref{eq:hdef}) belongs to $H^{1}(\mathbb{R}^{d}$), see Theorem 4.3.3, p.~152 of \cite{FukuOshiTake11} (due to the killing outside of $U$, the extended Dirichlet space is contained in $H^{1}(\mathbb{R}^{d})$), $h^{\eta}\to h$ in
$H^{1}(\mathbb{R}^{d}$), as $\eta\to0$. We thus find that	
\begin{equation}\label{2.23}
\lim_{\eta\to0}\mathcal{E}_{\mathbb{R}^{d}}(h^{\eta},h^{\eta})=\mathcal{E}_{\mathbb{R}^{d}}(h,h)=\mathrm{cap}_{\mathbb{R}^{d},U}(K^{2\delta}),
\end{equation}
where $\mathrm{cap}_{\mathbb{R}^{d},U}(K^{2\delta})$
is the relative capacity of $K^{2\delta}$ with respect to $U$,
and the last equality follows from \cite{FukuOshiTake11}, pp.~152 and 71.

\medskip
Letting $r_{U} \to\infty$ , the relative capacity converges to the usual Brownian capacity
(this follows for instance from the variational characterization of the capacity in Theorem 2.1.5 on pp.~70 and 71 of \cite{FukuOshiTake11}):
\begin{equation}\label{2.24}
\mathrm{cap}_{\mathbb{R}^{d},U}(K^{2\delta})\to\mathrm{cap}_{\mathbb{R}^{d}}(K^{2\delta}),  \;\mbox{as $r_U \rightarrow \infty$}.
\end{equation}

\n
Then, letting $\delta\to0$, by Proposition 1.13, p.~60 of \cite{PortSton78}, we find that
\begin{equation}\label{2.25}
\mathrm{cap}_{\mathbb{R}^{d}}(K^{2\delta})\to\mathrm{cap}_{\mathbb{R}^{d}}(K), \;\mbox{as $\delta \rightarrow 0$}.
\end{equation}
The claim (\ref{eq:limsup}) follows.
\end{proof}

\medskip\n

\begin{remark}\label{rem2.5} \rm

Our main objective in the next two sections is to prove (\ref{0.4}), i.e. $\widetilde{\mathbb{P}}_{N}[A_{N}]\to1$. Actually, we could also use the above $\widetilde{\mathbb{P}}_{N}$ (with $a>u$ in place of $u_{**}$ in the definition of $f$ in (\ref{eq:fdef})) and the change of probability method to provide an alternative proof of Theorem 6.4 of \cite{LiSzni13} (it derives the asymptotic lower bound for the probability that the regularized occupation-time profile of random interlacements insulates $K$ by values exceeding $a$). It is a remarkable feature that such a bulge of the occupation-time profile is constructed in the tilted interlacements by mostly steering the tilted walk towards $K_N$, and not by seriously tinkering the jump rates, see for instance (\ref{eq:generatordef}), as well as Propositions \ref{prop:Capcompare} and \ref{prop:mesdom} in the next section. \hfill $\square$
\end{remark}

\section{Domination of equilibrium measures}

In this section, our main goal is Proposition \ref{prop:mesdom}, where
we prove that on a mesoscopic box inside $K_{N}^{\delta}$, the tilted
equilibrium measure dominates $(u_{**}+\epsilon/4)/u$ times the corresponding standard equilibrium measure. It is the key ingredient for constructing
the coupling in Proposition \ref{prop:coupling} in the next section.
A major step is achieved in Proposition \ref{prop:Capcompare},
where we prove that the tilted capacity of a mesoscopic ball (larger than the above mentioned box) inside $K_{N}^{\delta}$ is at least $(u_{**}+\epsilon/2)/u$ times its corresponding standard capacity.

\medskip
We start with the precise definition of the objects of interest in
this and the next section. We denote by $\Gamma^{N}=\partial K_{N}^{\delta/2}$ the boundary in $\IZ^d$ of the discrete blow-up of $K^{\frac{\delta}{2}}$ (we
recall (\ref{eq:boundarydef}) and (\ref{eq:blowupdef-1}) for the definitions of the boundary
and of the discrete blow-up). The above $\Gamma^{N}$ will serve as a set ``surrounding'' $K_{N}$.
We fix numbers $r_{i}$, $i=1,\ldots,4$ such that
\begin{equation}\label{eq:r1234choice}
0<2r_{1}<r_{2}<r_{3}<r_{4}<1
\end{equation}

\n
We define for $x$ in $\Gamma^{N}$ two boxes centered at $x$ (when
there is ambiguity we add a superscript for its center $x$, and $B_2$ will only be used in Section 4):
\begin{equation}\label{eq:boxdef}
B_{1}=B_{\infty}(x,N^{r_{1}}),\qquad B_{2}=B_{\infty}(x,N^{r_{2}});
\end{equation}
and three balls also centered at $x$:
\begin{equation}\label{3.3}
B_{3}=B(x,N^{r_{3}}),\quad B_{4}=B(x,N^{r_{4}}),\quad B_{5}=B(x,2N^{r_{4}}),
\end{equation}
so that (in the notation of (\ref{eq:boundarydef})) one has
\begin{equation}\label{3.4}
B_{1}\subset B_{2}\subset B_{3}\subset B_{4}\subset B_{5}\subset\overline{B_{5}}\subseteq K_{N}^{\delta}\subset\subset\mathbb{Z}^{d}.
\end{equation}

\n
(we now tacitly assume that $N$ is sufficiently large so that for all
$x\in\Gamma^{N}$, $\overline{B_{5}^{x}}\subset K_{N}^{\delta}$, and the second equality of (\ref{eq:fconstant}) holds).

\medskip
We start with the domination of capacities. To prove the next
Proposition \ref{prop:Capcompare}, we calculate the time spent by the random
walk in the mesoscopic body $B_{3}$ in two different ways (see
Lemma \ref{lem:MEAN}), and relate these expressions to the
capacity and to the tilted capacity.
\begin{prop}\label{prop:Capcompare}When $N$ is large, we have for all $x\in\Gamma^N$
\begin{equation}\label{eq:Capcompare}
u\mathrm{\widetilde{c}ap}(B_{3})\geq\Big(u_{**}+\frac{\epsilon}{2}\Big)\mathrm{cap}(B_{3}).
\end{equation}
\end{prop}

The proof of this proposition relies on Lemmas \ref{lem:MEAN} and
\ref{lem:pmax}.
\begin{lem}\label{lem:MEAN}
\begin{equation}\label{eq:MEAN}
\widetilde{E}_{\widetilde{e}_{B_{3}}}\Big[\dis\int_{0}^{\infty}1_{B_{3}}(X_{s})ds\Big]=\frac{u_{**}+\epsilon}{u}\, E_{e_{B_3}}\Big[\int_{0}^{\infty}1_{B_{3}}(X_{s})ds\Big]
\end{equation}
\end{lem}
\begin{proof}
By the definition of the tilted Green function (see (\ref{eq:tiltedgdef}))
and by (\ref{eq:eKg1tilted}),
\begin{equation}\label{3.7}
\begin{array}{lcl}
\widetilde{E}_{\widetilde{e}_{B_3}}\Big[\int_{0}^{\infty}1_{B_{3}}(X_{s})ds\Big] & \!\!\!\!= & \!\!\!\!\sum_{v\in\partial_{i}B_3,\: y\in B_{3}}\widetilde{e}_{B_3}(v)\widetilde{g}(v,y)\widetilde{\lambda}(y)\\
 & \!\!\!\! \overset{(\ref{eq:eKg1tilted})}{=} & \!\!\!\! \sum_{y\in B_{3}}1_{B_{3}}(y)\widetilde{\lambda}(y).
\end{array}
\end{equation}
Moreover, $\widetilde{\lambda}(y)=f^{2}(y)=\frac{u_{**}+\epsilon}{u}$
for $y\in B_{3}\subset K_{N}^{\delta}$ (see (\ref{eq:lambdadef}), (\ref{eq:fconstant}), (\ref{3.4})). Hence,
\begin{equation}\label{3.8}
\widetilde{E}_{\widetilde{e}_{B_3}}\Big[\int_{0}^{\infty}1_{B_{3}}(X_{s})ds\Big]=\frac{u_{**}+\epsilon}{u}|B_{3}|.
\end{equation}
By a similar calculation, we also find that
\begin{equation}\label{3.9}
\mathbb{E}_{e_{B_3}}\Big[\int_{0}^{\infty}1_{B_{3}}(X_{s})ds\Big]=|B_{3}|.
\end{equation}

\n
Comparing  (\ref{3.8}) and (\ref{3.9}) , we obtain (\ref{eq:MEAN}) as desired.\end{proof}

In the second lemma we prove that starting from the boundary of $B_{4}$,
the tilted walk hits $B_{3}$ with a probability tending to 0 with $N$.
\begin{lem}
\label{lem:pmax}
\begin{equation}\label{eq:pmax}
\beta(N)\overset{\mathrm{def}}{=}\max_{x\in\Gamma^N,v\in\partial B_{4}}\widetilde{P}_{v}(H_{B_{3}}<\infty)\textrm{ tends to 0 as }N\to\infty.
\end{equation}
\end{lem}
\begin{proof}
For $v$ in $\partial B_{4}$, we have
\begin{equation}\label{eq:est0}
\widetilde{P}_{v}(H_{B_{3}}<\infty)=\widetilde{P}_{v}(H_{B_{3}}<T_{B_{5}})+\widetilde{P}_{v}(T_{B_{5}}<H_{B_{3}}<\infty),
\end{equation}

\n
By the second equality of (\ref{eq:fconstant}), and in view of (\ref{eq:generatordef}), (\ref{3.4}),
when starting in $v\in B_{4}$, under $\widetilde{P}_{v}$, $X_{\cdot\wedge T_{B_{5}}}$
behaves as stopped simple random walk. Thus, by classical simple random
walk estimates, we have an upper bound for the
probability that the tilted walk hits $B_{3}$ before exiting $B_{5}$:
\begin{equation}\label{eq:est1}
\max_{v\in\partial B_{4}}\widetilde{P}_{v}(H_{B_{3}}<T_{B_{5}})\leq\max_{v\in\partial B_{4}}P_{v}(H_{B_{3}}<\infty)\overset{\mathrm{def}}{=}\beta_{0}(N)=O(N^{(r_{3}-r_{4})(d-2)}).
\end{equation}
(note that $\beta_0 (N)$ does not depend on $x\in \Gamma^N$).

\medskip
By the strong Markov property successively applied  at times $T_{B_{5}}$
and $H_{\overline{B_{4}}}$, we have:
\begin{equation}\label{eq:est2}
\widetilde{P}_{v}(T_{B_{5}}<H_{B_{3}}<\infty)\leq\max_{y\in\partial B_{5}}\widetilde{P}_{y}(H_{\overline{B_{4}}}<\infty)\max_{v'\in\partial B_{4}}\widetilde{P}_{v'}(H_{B_{3}}<\infty).
\end{equation}

\n
Taking the maximum over $v$ in $\partial B_{4}$ on the left-hand
side of (\ref{eq:est2}), and inserting this bound in (\ref{eq:est0}),
we find with the help of (\ref{eq:est1}):
\begin{equation}\label{3.14}
\max_{v\in\partial B_{4}}\widetilde{P}_{v}(H_{B_{3}}<\infty)\leq\frac{\beta_{0}(N)}{1-\max\limits_{y\in\partial B_{5}}\widetilde{P}_{y}(H_{\overline{B_{4}}}<\infty)}.
\end{equation}

\n
To prove (\ref{eq:pmax}), it now suffices to show that
\begin{equation}\label{eq:leaveforever}
\liminf_{N}\min_{x\in\Gamma^N,y\in\partial B_{5}}\widetilde{P}_{y}(H_{\overline{B_{4}}}=\infty)>0.
\end{equation}

\n
As a result of (\ref{eq:greenasym}) and the stopping theorem, for
large $N$, and any $x\in\Gamma^N$,
\begin{equation}\label{eq:leaveproba}
\min_{y\in\partial B_{5}}P_{y}(H_{\overline{B_{4}}}=\infty)>c.
\end{equation}

\n
By a similar argument as in Lemma \ref{lem:MEA2},
\begin{equation}\label{eq:timecontrol}
E_{z}\Big[\int_{0}^{\infty}1_{\widetilde{U}_N}(X_{s})ds\Big]\leq c(\widetilde{U})N^{2}\textrm{, for }z\in\mathbb{Z}^{d}\textrm{ and }N\geq1.
\end{equation}

\n
By the Chebyshev Inequality, writing $\widetilde{c}(\widetilde{U})=2c(\widetilde{U})/c$, with $c$ as in (\ref{eq:leaveproba}), and $I_{N}=\{\int_{0}^{\infty}1_{\widetilde{U}_N}(X_{s})ds\leq\widetilde{c}(\widetilde{U})N^{2}\}$,
we have
\begin{equation}\label{eq:staytime}
P_{z}[I_{N}]\geq1-\frac{c}{2}\textrm{, for all }z\in\mathbb{Z}^{d}.
\end{equation}

\n
With (\ref{eq:leaveproba}) and (\ref{eq:staytime}) put together,
we obtain that for all $z$ in $\partial B_{5}$,
\begin{equation}\label{eq:newleaveproba}
P_{z}(\{H_{\overline{B_{4}}}=\infty\}\cap I_{N})\ge\frac{c}{2}.
\end{equation}

\n
By definition of $f$ (see (\ref{eq:fdef})) and since $h^{\eta}\in C_{0}^{\infty}$,
we see that
\begin{equation}\label{eq:Uupperbound}
|V|=\Big|\frac{\Delta_{dis}f}{f}\Big|\leq c(u) \Big|\Delta_{dis}h_{N}\Big|\leq\frac{\bar{c}(h^{\eta},u)}{N^{2}}.
\end{equation}

\n
By the first equality of (\ref{eq:fconstant}), we have $\Delta_{dis}f=0$ outside $\widetilde{U}_N$. Hence, we find that for large $N$, for all $x \in \Gamma^N$ and $y \in \partial B_5$, on the event $I_{N}$,
\begin{equation}\label{eq:tiltedlowerboudn}
\frac{d\widetilde{P}_{y}}{dP_{y}} \ge c(u)\exp\Big\{\int_{0}^{\infty}V(X_{s})ds\Big\}\overset{(\ref{eq:timecontrol})}{\underset{(\ref{eq:Uupperbound})}{\geq}} c(u)\exp\Big\{-\widetilde{c}N^{2}\cdot\frac{\bar{c}}{N^{2}}\Big\}= c(u)e^{-\widetilde{c}\bar{c}}.
\end{equation}
Therefore, by (\ref{eq:newleaveproba}, (\ref{eq:tiltedlowerboudn}) we find that
\begin{equation}\label{3.22}
\begin{array}{l}
\liminf_{N\to\infty}\min_{x\in\Gamma^N,y\in\partial B_{5}}\widetilde{P}_{y}[\{H_{\overline{B_4}}\, =\infty\}]   \geq 
\\[2ex]
\liminf_{N\to\infty}\min_{x\in\Gamma^N,y\in\partial B_{5}}E_{y}\Big[\dis\frac{d\widetilde{P}_{y}}{dP_{y}}1_{\{H_{\overline{B_4}}\,=\infty\}},I_{N}\Big]>0. 
\end{array}
\end{equation}

\medskip\n
This proves (\ref{eq:leaveforever}) and concludes the proof of Lemma
\ref{lem:pmax}.
\end{proof}

\medskip
With all ingredients prepared, we are ready to prove the domination
of capacities stated in Proposition \ref{prop:Capcompare}. In the proof we combine the estimates obtained in Lemmas
\ref{lem:MEA2} and \ref{lem:MEAN}, perform an argument similar to
(\ref{eq:est0}), (\ref{eq:est1}) and (\ref{eq:est2}), and employ
Lemma \ref{lem:pmax} to control the tilted return probability.
\begin{proof}[Proof of Proposition \ref{prop:Capcompare}]
We will bound the left term of (\ref{eq:MEAN}) from above and the right term from below. We start with the upper bound on the left-hand side of (\ref{eq:MEAN}).

\medskip\n
For all $y$ in $\partial_{i}B_{3}$, by strong Markov property at
time $T_{B_{4}}$ (and then at time $H_{B_3}$) we have
\begin{equation}\label{eq:telescope}
\begin{split}
\widetilde{E}_{y}\Big[\int_{0}^{\infty} \! 1_{B_{3}}(X_{s})ds\Big]  = &\, \widetilde{E}_{y}\Big[\int_{0}^{T_{B_{4}}}\! 1_{B_{3}}(X_{s})ds\Big]+\widetilde{E}_{y}\Big[\widetilde{E}_{X_{T_{B_{4}}}}\Big[\int_{0}^{\infty} \! 1_{B_{3}}(X_{s})ds\Big]\Big]
\\[2ex]
  \leq & \;\max_{y\in\partial_{i}B_{3}}\Big\{\widetilde{E}_{y}\Big[\int_{0}^{T_{B_{4}}}1_{B_{3}}(X_{s})ds\Big]\Big\}
 \\[2ex]
  + &\; \max_{v\in\partial B_{4}}\{\widetilde{P}_{v}[H_{B_{3}}<\infty]\}\max_{y\in\partial_{i}B_{3}}\{\widetilde{E}_{y}\Big[\int_{0}^{\infty}1_{B_{3}}(X_{s})ds\Big].
\end{split}
\end{equation}

\medskip\n
Taking the maximum over $y\in\partial_{i}B_{3}$ on the left-hand
side of (\ref{eq:telescope}) and rearranging, we find in view of
(\ref{eq:pmax}):
\begin{equation}\label{eq:maxineq}
\max_{y\in\partial_{i}B_{3}}\widetilde{E}_{y}\Big[\int_{0}^{\infty}1_{B_{3}}(X_{s})ds\Big]\leq\frac{\max_{y\in\partial_{i}B_{3}}\widetilde{E}_{y}\big[\int_{0}^{T_{B_{4}}}1_{B_{3}}(X_{s})ds\big]}{1-\beta(N)}.
\end{equation}

\medskip\n
Then we notice that, since $f$ is constant on $K_{N}^{\delta}\supseteq\overline{B_{4}}$, see (\ref{eq:fconstant}) and (\ref{3.4}),
\begin{equation}\label{eq:tiltequiv}
\widetilde{E}_{y}\Big[\dis\int_{0}^{T_{B_{4}}}1_{B_{3}}(X_{s})ds\Big]=E_{y}\Big[\dis\int_{0}^{T_{B_{4}}}1_{B_{3}}(X_{s})ds\Big].
\end{equation}
 We now have the following upper bound on the left-hand side of (\ref{eq:MEAN})
under $\widetilde{P}_{\widetilde{e}_{B_{3}}}$:
\begin{equation}\label{eq:capcomp}
\begin{array}{lcl}
\widetilde{E}_{\widetilde{e}_{B_{3}}}\Big[\dis\int_{0}^{\infty}1_{B_{3}}(X_{s})ds\Big] &
\!\!\!\! \leq &\!\!\!\!  \mathrm{\widetilde{c}ap}(B_{3})\;\max_{y\in\partial_{i}B_{3}}\widetilde{E}_{y}\Big[\dis\int_{0}^{\infty}1_{B_{3}}(X_{s})ds\Big]
\\[4ex]
 & \!\!\!\! \overset{(\ref{eq:maxineq})}{\leq} &\!\!\!\!  \mathrm{\widetilde{c}ap}(B_{3})\; \dis\frac{\max_{y\in\partial_{i}B_{3}}\big\{\widetilde{E}_{y}\big[\int_{0}^{T_{B_{4}}}1_{B_{3}}(X_{s})ds\big]\big\}}{1-\beta(N)}
\\[4ex]
 & \!\!\!\! \overset{(\ref{eq:tiltequiv})}{=} &\!\!\!\!  \mathrm{\widetilde{c}ap}(B_{3}) \; \dis\frac{\max_{y\in\partial_{i}B_{3}}\big\{E_{y}\big[\int_{0}^{T_{B_{4}}}1_{B_{3}}(X_{s})ds\big]\big\}}{1-\beta(N)}
\\[4ex]
 &\!\!\!\!  \leq & \!\!\!\! \mathrm{\widetilde{c}ap}(B_{3})\; \dis\frac{\max_{y\in\partial_{i}B_{3}}\big\{E_{y}\big[\int_{0}^{\infty}1_{B_{3}}(X_{s})ds\big]\big\}}{1-\beta(N)}
\\[4ex]
 &\!\!\!\! \overset{(\ref{eq:alpha})}{\leq} &\!\!\!\!  \mathrm{\widetilde{c}ap}(B_{3})c_{1}N^{2r_{3}}\; \dis\frac{1+\alpha(N)}{1-\beta(N)}.
\end{array}
\end{equation}

\medskip\n
On the other hand, by (\ref{eq:alpha}) of Lemma \ref{lem:MEA2}, we have a lower bound on the right-hand side of (\ref{eq:MEAN}):
\begin{equation}\label{eq:equa}
\frac{u_{**}+\epsilon}{u}E_{e_{B_{3}}}\Big[\int_{0}^{\infty}1_{B_{3}}(X_{s})ds\Big]\geq\frac{u_{**}+\epsilon}{u}\mathrm{cap}(B_{3})c_{1}N^{2r_{3}}(1-\alpha(N)).
\end{equation}

\medskip\n
Combining (\ref{eq:capcomp}), (\ref{eq:equa}) and Lemma \ref{lem:MEAN}, we find
\begin{equation}
\mathrm{\widetilde{c}ap}(B_{3})\frac{1+\alpha(N)}{1-\beta(N)}\geq\frac{u_{**}+\epsilon}{u}\mathrm{cap}(B_3)(1-\alpha(N)).
\end{equation}
With the help of (\ref{eq:alpha}) and (\ref{eq:pmax}) we see that Proposition \ref{prop:Capcompare} readily follows.
\end{proof}

We now turn to the domination of the equilibrium measures at a smaller
scale on $B_{1}$. In the proof of Proposition \ref{prop:mesdom},
thanks to the domination of capacities proved in Proposition \ref{prop:Capcompare},
we are able to reduce the domination of equilibrium measures to the
domination of (relative) entrance measures. This is performed in Lemma
\ref{lem:surfacedom}.
\begin{prop}
\label{prop:mesdom}When $N$ is large, for all $x\in\Gamma^N$ and $z\in\partial_{i}B_{1},$
\begin{equation}\label{eq:mesdom}
u\widetilde{e}_{B_{1}}(z)\geq\Big(u_{**}+\frac{\epsilon}{4}\Big)e_{B_{1}}(z).
\end{equation}
\end{prop}

The proof of Proposition \ref{prop:mesdom} relies on the following
lemma, where we prove that the killed entrance measure of $B_{1}$
almost dominates the corresponding standard entrance measure. From now on, we fix $\epsilon'=\epsilon/(4u_{**}+2\epsilon)$. We recall (\ref{eq:hdefsf})
for notation.
\begin{lem}\label{lem:surfacedom} For sufficiently large $N$, for all $x\in\Gamma^N$ and $z\in\partial_{i}B_{1},$
\begin{equation}\label{eq:surfacedom}
\min_{y\in\partial_{i}B_{3}}h_{B_{1},B_{4}}(y,z)\geq(1-\epsilon')\max_{\widetilde{y}\in\partial_{i}B_{3}}h_{B_{1}}(\widetilde{y},z).
\end{equation}
\end{lem}

\medskip
The proof of Lemma \ref{lem:surfacedom} has the same flavour as Section
3 of \cite{BenjSzni08} and indeed relies on Lemma 3.3 of the
same reference.
\begin{proof}
We decompose $h_{B_{1,}B_{4}}(y,z)$ according to the time and place
of the last step before entering $B_{1}$ at $z$, and obtain for
$y$ outside $B_{1}$ and $z$ in $B_{1}$

\begin{equation}\label{3.34}
h_{B_{1},B_{4}}(y,z)=\frac{1}{2d}\sum_{z'\sim z,z'\in\partial B_{1}}g_{B_{4}\backslash B_{1}}(y,z').
\end{equation}

\n
Similarly, we have for $\widetilde{y}$ outside $B_{1}$ and $z$ in $B_{1}$,
\begin{equation}\label{3.35}
h_{B_{1}}(\widetilde{y},z)=\frac{1}{2d}\sum_{z'\sim z,z'\in\partial B_{1}}g_{B_{1}^{c}}(\widetilde{y},z').
\end{equation}

\n
Therefore, to prove (\ref{eq:surfacedom}), it suffices to show that for
large $N$ and for all $y,\widetilde{y}\in\partial_{i}B_{3}$ and $z'\in\partial B_{1}$
\begin{equation}
g_{B_{4}\backslash B_{1}}(y,z')\geq(1-\epsilon')g_{B_{1}^{c}}(\widetilde{y},z').\label{eq:greenineq}
\end{equation}

\medskip\n
By an argument similar to Lemma 3.3 of \cite{BenjSzni08} to $B_{4}$
and $B_{1}$, we have that
\begin{equation}\label{3.37}
\begin{array}{lcl}
g_{B_{4}\backslash B_{1}}(y,z') & \!\!\! \overset{\textrm{symmetry}}{=} &  \! g_{B_{4}\backslash B_{1}}(z',y)
\\[2ex]
 & \!\!\!  \overset{\textrm{Markov}}{=} & \!  g_{B_{4}}(z',y)-E_{z'}[g_{B_{4}}(X_{H_{B_{1}}},y),\ H_{B_{1}}<T_{B_{4}}]
 \\[2ex]
 &\!\! \stackrel{\textrm{symmetry}}{=} & \! E_{z'}[g_{B_{4}}(y,z')-g_{B_{4}}(y,X_{H_{B_{1}}}),\ H_{B_{1}}<T_{B_{4}}]
 \\[2ex]
 &  \!\!\!  &  \! +~g_{B_{4}}(y,z')P_{z'}[H_{B_{1}}>T_{B_{4}}]\:\overset{\mathrm{def}}{=}\: A+B.
\end{array}
\end{equation}

\medskip\n
Then, by the gradient estimate and the Harnack inequality in Theorems 1.7.1, and 1.7.2,  p.~42 of \cite{Lawl91},
\begin{equation}\label{3.38}
|A|\leq\frac{c}{N^{r_{3}}}N^{r_{1}}g_{B_{4}}(y,z'),
\end{equation}
and by a similar argument as below (3.30) of \cite{BenjSzni08},
\begin{equation}\label{3.39}
B\geq\frac{c}{N^{r_{1}}}g_{B_{4}}(y,z').
\end{equation}
Hence, collecting (\ref{3.37}), (\ref{3.38}), (\ref{3.39}), we find that
\begin{equation}\label{eq:gest1}
g_{B_{4}\backslash B_{1}}(y,z')\geq g_{B_{4}}(y,z')P_{z'}[H_{B_{1}}>T_{B_{4}}](1-cN^{2r_{1}-r_{3}}),
\end{equation}
By analogous arguments we also obtain
\begin{equation}\label{eq:gest2}
g_{B_{1}^{c}}(\widetilde{y},z')\le g(\widetilde{y},z')P_{z'}[H_{B_{1}}=\infty](1+cN^{2r_{1}-r_{3}}).
\end{equation}

\medskip\n
By the definition of $r_{1}$ and $r_{3}$ (see (\ref{eq:r1234choice})),
$N^{2r_{1}-r_{3}}\ll1$. Therefore, combining (\ref{eq:gest1}),
(\ref{eq:gest2}) together with the fact that
\begin{equation}\label{3.42}
P_{z'}[H_{B_{1}}>T_{B_{4}}]\geq P_{z'}[H_{B_{1}}=\infty],
\end{equation}

\medskip\n
the claim (\ref{eq:greenineq}) will follow once we show (see above Lemma \ref{lem:surfacedom} for our choice of $\epsilon'$) that when $N$ is sufficiently large, for all $x \in \Gamma^N$, all $y,\widetilde{y}\in\partial_{i}B_{3}$ and
all $z'\in\partial B_{1}$,
\begin{equation}\label{eq:gcomp}
g_{B_{4}}(y,z')\geq\Big(1-\frac{\epsilon'}{2}\Big)\,g(\widetilde{y},z').
\end{equation}

\medskip\n
By (\ref{eq:greenasym}) and (\ref{eq:gUestimate}), for large $N$, setting $\widetilde{B} = B(y, \frac{N^{r_4}}{2})$ we have the following bounds:
\begin{equation}\label{eq:g6}
g_{B_{4}}(y,z')\,\geq g_{\widetilde{B}} (y,z') \geq\bar{c}_{0} |y - z'|^{(2-d)}-cN^{r_{4}(2-d)}- c'N^{r_{3}(1-d)}
\end{equation}
and
\begin{equation}\label{eq:g7}
g(\widetilde{y},z')\leq\bar{c}_{0}|y- z'|^{(2-d)}+cN^{r_{3}(1-d)}.
\end{equation}

\medskip\n
Hence, we obtain (\ref{eq:gcomp}) and (\ref{eq:greenineq}) follows.
This proves Lemma \ref{lem:surfacedom}.
\end{proof}

\medskip
We are now ready to prove Proposition \ref{prop:mesdom}. In the proof,
we make use of the sweeping identity, and, in effect, reduce the comparison
of the standard and tilted equilibrium measures of $B_{1}$ to the comparison on the standard and tilted capacities of $B_{3}$, and to the comparison of the (killed)
entrance measures.
\begin{proof}[Proof of Proposition \ref{prop:mesdom}]
For large $N$ and for all $x\in\Gamma^N$ and $z\in\partial_{i}B_{1},$ we find that
\begin{equation}\label{eq:halfsurfacecontrol}
\begin{array}{lcl}
u\widetilde{e}_{B_{1}}(z) &\!\!\! \overset{\textrm{(\ref{eq:sweeping-1})}}{=} &\!\!\!  u\,\widetilde{P}_{\widetilde{e}_{B_{3}}}(X_{H_{B_{1}}}=z,\, H_{B_{1}}<\infty)
\\[2ex]
 &\!\!\!   \geq &\!\!\!  u\, \mathrm{\widetilde{c}ap}(B_{3})\min_{y\in\partial_{i}B_{3}}\widetilde{h}_{B_{1}}(y,z)
 \\[2ex]
 &\!\!\! \overset{(\ref{eq:Capcompare})}{\geq} &\!\!\!  \Big(u_{**}+\dis\frac{\epsilon}{2}\Big)\;\mathrm{cap}(B_{3})\min_{y\in\partial_{i}B_{3}}\widetilde{h}_{B_{1}}(y,z)
 \\[3ex]
 &\!\!\!  \geq & \!\!\! \Big(u_{**}+\dis\frac{\epsilon}{2}\Big)\;\mathrm{cap}(B_{3})\min_{y\in\partial_{i}B_{3}}\widetilde{h}_{B_{1},B_{4}}(y,z).
\end{array}
\end{equation}

\medskip\n
Since up to the exit time from $B_{4}$ the tilted and standard walk
have the same law (see (\ref{eq:fconstant})), we see that for $y\in\partial_{i}B_{3}$ and $z\in\partial B_{1}$, we have
\begin{equation}\label{eq:hsrw}
\widetilde{h}_{B_{1},B_{4}}(y,z)=h_{B_{1},B_{4}}(y,z).
\end{equation}

\medskip\n
Taking Lemma \ref{lem:surfacedom} into account, we find that for
large $N$ and for all $x\in\Gamma^N$ and $z\in\partial B_{1}$,
\begin{equation}
\min_{y\in\partial_{i}B_{3}}h_{B_{1},B_{4}}(y,z)\overset{(\ref{eq:surfacedom})}{\geq}(1-\epsilon')\max_{\widetilde{y}\in\partial_{i}B_{3}}h_{B_{1}}(\widetilde{y},z).
\end{equation}

\n
Thus, coming back to (\ref{eq:halfsurfacecontrol}), we find that
with our choice of $\epsilon'$ (above Lemma \ref{lem:surfacedom}),
\begin{equation}\label{3.48}
\begin{array}{lcl}
u\widetilde{e}_{B_{1}}(z) &\!\!\!\! \geq &\!\!\!\! \Big(u_{**}+ \dis \frac{\epsilon}{4}\Big)\mathrm{cap}(B_{3})\max_{\widetilde{y}\in\partial_{i}B_{3}}h_{B_{1}}(\widetilde{y},z)
\\[2ex]
 &\!\!\!\! \overset{(\ref{eq:hdefsf})}{\ge} & \!\!\!\! \Big(u_{**}+\dis \frac{\epsilon}{4}\Big)\, P_{e_{B_{3}}}(X_{H_{B_{1}}}=z,\, H_{B_{1}}<\infty)
 \\[2ex]
 &\!\!\!\! \stackrel{(\ref{eq:sweeping})}{=} & \!\!\!\! \Big(u_{**}+\dis \frac{\epsilon}{4}\Big)\, e_{B_{1}}(z).
 \end{array}
 \end{equation}

 \medskip\n
This completes the proof of Proposition \ref{prop:mesdom}.
\end{proof}

\section{Coupling and Disconnection}

In this section, we prove in Theorem \ref{thm:tiltedone} that the
tilted interlacements disconnect $K_{N}$ from infinity with a probability,
which tends to 1 as $N$ goes to infinity. To this end, we show that
in mesoscopic boxes with centers in $\Gamma^{N}$ (introduced above (\ref{eq:r1234choice})), the tilted random interlacements
locally ``dominate'' random interlacements with level higher than $u_{**}$,
and thus typically disconnect in each such box the center from its boundary with
very high probability. Therefore, there is a high probability as well
for the tilted interlacement to disconnect the macroscopic body from
infinity. The main step is Proposition \ref{prop:coupling} where we construct at each point of $\Gamma^N$ a coupling so that the tilted random interlacements with high probability locally dominate some standard random interlacements with level higher than $u_{**}$. 

\medskip
We recall the definitions of $B_{1}$ and \textbf{$B_{2}$} from (\ref{eq:boxdef}).

\begin{prop} \label{prop:coupling} When $N$ is large, for all $x\in\Gamma^{N}$,
there exists a probability space $(\bar{\Omega},\bar{\mathcal{A}},\bar{Q})$
and random sets $\widetilde{\mathcal{I}}$ and $\mathcal{I}_{1}$ defined
on $\bar{\Omega}$, with same respective laws as $\mathcal{I}^{u}\cap B_{1}$
under $\widetilde{\mathbb{P}}_{N}$ and $\mathcal{I}^{u_{**}+\frac{\epsilon}{8}}$
under $\mathbb{P}_{u_{**}+\frac{\epsilon}{8}}$, so that
\begin{equation}\label{eq:coupling}
\bar{Q}[\widetilde{\mathcal{I}}\supset\mathcal{I}_{1}]\geq1-c_{5}e^{-c_{6}N^{c_{7}}}
\end{equation}

\n
(the constants depend on $r_1, r_2, \epsilon$).
\end{prop}

\medskip
The idea of the proof is to stochastically dominate the trace in $B_{1}$
of  random interlacements with level higher than $u_{**}$ by the ``first excursions''
(from some inner boundary of $B_{1}$ to $\partial B_{2}$) of the
trajectories from some random interlacements with slightly higher intensity,
and then, further dominate these excursions by the same kind of ``first
excursions'' of trajectories of the tilted interlacement. The following
proposition for the above mentioned first stochastic domination in essence rephrases Proposition 4.4 of \cite{Beli13}. We begin with some notation.

\medskip
For $A\subset B\subset\subset\mathbb{Z}^{d}$, we write $k_{A,B}$
for the law on $\Gamma(\IZ^d)$ (see below (\ref{1.4})) of the stopped process $X_{\cdot\wedge T_{B}}$ under $P_{e_{A}}$. We also denote the trace of a point process $\eta=\sum_{i}\delta_{w_{i}}$ on the space $\Gamma (Z^d)$ by 
\begin{equation}\label{4.2}
\mathcal{I}(\eta)=\cup_{i}Range(w_{i}).
\end{equation}

\begin{prop}\label{prop:coupinter} When $N$ is large, for all $x \in \Gamma^N$, there exists a probability
space $(\Sigma,\mathcal{B},Q)$ endowed with a Poisson point process
$\eta$, with intensity measure $(u_{**}+\epsilon/4)k_{B_{1},B_{2}}$,
and a random set $\mathcal{I}_{1}\subset\mathbb{Z}^{d}$ with the
law of \textup{$\mathcal{I}^{u_{**}+\frac{\epsilon}{8}}\cap B_{1}$
under $\mathbb{P}_{u_{**}+\frac{\epsilon}{8}}$, and
\begin{equation}\label{4.3}
Q[\mathcal{I}_{1}\subset\mathcal{I}(\eta)\cap B_{1}]\geq1-c_{5}e^{-c_{6}N^{c_{7}}}.
\end{equation}
}
\end{prop}

We refer the readers to Proposition 5.4 of \cite{Beli13} and to Section
8 of \cite{Beli13} for the proof of Proposition \ref{prop:coupinter}.

\medskip
We now construct another coupling such that the trace on $B_1$ of the first
excursions of the tilted random interlacements dominate the trace of the corresponding excursions for random interlacements at level $u_{**}+\frac{\epsilon}{4}$. Combined
with Proposition \ref{prop:coupinter}, this will complete the proof of
Proposition \ref{prop:coupling}.

\begin{proof}[Proof of Proposition \ref{prop:coupling}]
We keep the notation of Proposition \ref{prop:coupinter}. Let $\alpha$
be the measure on $\partial_{i}B_{1}$ such that for all $z\in\partial_{i}B_{1}$,
\begin{equation}
\alpha(z)=u\widetilde{e}_{B_{1}}(z)- \Big(u_{**}+\frac{\epsilon}{4}\Big)\,e_{B_{1}}(z).
\end{equation}

\n
By Proposition \ref{prop:mesdom} $\alpha$ is a positive measure.
Hence, we can construct an auxiliary probability space $(\widetilde{\Omega},\widetilde{\mathcal{A}},\widetilde{Q})$, endowed with a Poisson point process $\widetilde{\eta}$ on $\Gamma(\mathbb{Z}^{d})$ with intensity measure $k_{\alpha}(\cdot)=P_{\alpha}(X_{\cdot\wedge T_{B_{2}}})$.
Since for all $z$ in $\partial_{i}B_{1}$, the tilted walk coincides with the simple random walk up to the exit from $B_{2}$, we obtain that
\begin{equation}\label{4.5}
\widetilde{\mathcal{I}}=(\mathcal{I}(\widetilde{\eta})\cup\mathcal{I}(\eta))\cap B_{1}\textrm{ is stochastically dominated by }\mathcal{I}^{u}\cap B_{1}\textrm{ under }\widetilde{\mathbb{P}}_{N}.
\end{equation}
We can thus construct on some extension $(\overline{\Omega}, \overline{{\cal A}}, \overline{Q})$ an $\widetilde{{\cal I}}$ distributed as ${\cal I}^u \cap B_1$ under $\widetilde{\IP}_N$, so that $\widetilde{\cal I} \supseteq {\cal I}(\eta)$, $\overline{Q}$-a.s.. We then have
\begin{equation}\label{4.6}
\overline{Q}[\widetilde{\mathcal{I}}\supset\mathcal{I}_{1}]\geq\overline{Q}[\mathcal{I}(\eta)\cap B_{1}\supset\mathcal{I}_{1}]=Q[\mathcal{I}(\eta)\cap B_{1}\supset\mathcal{I}_{1}]\stackrel{(\ref{4.3})}{\geq} 1-c_{5}e^{-c_{6}N^{c_{7}}}.
\end{equation}
\end{proof}

\medskip
We are now ready to derive a key step for the proof of Theorem
\ref{thm:lowerbound}. Namely, we will now show that with $\widetilde{\mathbb{P}}_{N}$-probability
tending to $1$, the event $A_{N}$($=\{K_{N}\; \overset{\mathcal{V}^{u}}{\mbox{\Large $\nleftrightarrow$}}\; \infty\}$, see (\ref{eq:ANdef})) does occur.
\begin{thm} \label{thm:tiltedone}
\begin{equation}\label{eq:tiltedconv}
\lim_{N\to\infty}\widetilde{\mathbb{P}}_{N}[A_{N}]=1.
\end{equation}
\end{thm}

\begin{proof}
Note that for large $N$, when $K_{N}$ is connected to infinity by a nearest-neighbor
path, this path must go through the set $\Gamma^{N}$ at some point
$x$ (see above (\ref{eq:r1234choice})). Hence, this path connects $x$ to the inner boundary of $B_{1}^{x}$, so that
\begin{equation}\label{4.8}
A_{N}^{c}\subset\cup_{x\in\Gamma^{N}}\{x\overset{\mathcal{V}^{u}}{\longleftrightarrow}\partial_{i}B_{1}^{x}\}.
\end{equation}
Thus, we find that for large $N$
\begin{equation}\label{eq:probdecomp}
\widetilde{\mathbb{P}}_{N}[A_{N}^{c}]\leq\sum_{x\in\Gamma^{N}}\widetilde{\mathbb{P}}_{N}[x\overset{\mathcal{V}^{u}}{\longleftrightarrow}\partial_{i}B_{1}^{x}].
\end{equation}

\n
By Proposition \ref{prop:coupling}, for large $N$, uniformly in $x \in \Gamma^N$, we can bound the probability in the right-hand side of (\ref{eq:probdecomp}) as follows,
\begin{equation}
\begin{array}{lcl}
\widetilde{\mathbb{P}}_{N}[x\overset{\mathcal{V}^{u}}{\leftrightarrow}\partial_{i}B_{1}^{x}] & \!\!\! \overset{(\ref{eq:coupling})}{\leq} &  \!\!\!\mathbb{P}_{u_{**}+\frac{\epsilon}{8}}[x\overset{\mathcal{V}^{u_{**}+\frac{\epsilon}{8}}}{\longleftrightarrow}\partial_{i}B_{1}^{x}]+c_{5}e^{-c_{6}N^{c_{7}}}.
\\
 & \!\!\! \overset{(\ref{eq:supercrit})}{\leq} & \!\!\! ce^{-c'N^{\widetilde{c}}},
\end{array}
\end{equation}

\medskip\n
where the constants depend on $r_1, r_2, \epsilon$.

\medskip
Hence, we see that for large $N$,
\begin{equation}\label{4.12}
\widetilde{\mathbb{P}}_{N}[A_{N}^{c}]\leq|\Gamma^{N}|ce^{-c'N^{\widetilde{c}}}\underset{N}{\longrightarrow} 0.
\end{equation}

\medskip\n
This concludes the proof of Theorem \ref{thm:tiltedone}.
\end{proof}

\section{Denouement}

In this section we combine the various ingredients, namely Theorem \ref{thm:tiltedone},
Propositions \ref{prop:entropycalc} and \ref{prop:limsup}, and prove Theorem \ref{thm:lowerbound}.
\begin{proof}[Proof of Theorem \ref{thm:lowerbound}]
We recall the entropy inequality (see (\ref{eq:Entropychange})),
and apply it to $\mathbb{P}_{u}$ and $\widetilde{\mathbb{P}}_{N}$
defined in Sections 1 and 2. By Theorem \ref{thm:tiltedone}, we know
that  \\
$\lim_{N\to\infty}\widetilde{\mathbb{P}}_{N}[A_{N}]=1$, and (\ref{eq:Entropychange}) yields that
\begin{equation}\label{eq:entropychangeappli}
\liminf_{N\to\infty}\frac{1}{N^{d-2}}\log(\mathbb{P}_{u}[A_{N}])\geq-\limsup_{N\to\infty}\frac{1}{N^{d-2}}H(\widetilde{\mathbb{P}}_{N}|\mathbb{P}_{u}).
\end{equation}

\medskip\n
By Proposition \ref{prop:entropycalc}, we represent the right-hand
side of (\ref{eq:entropychangeappli}) as
\begin{equation}\label{5.2}
-\limsup_{N\to\infty}\frac{1}{N^{d-2}}H(\widetilde{\mathbb{P}}_{N}|\mathbb{P}_{u})=-(\sqrt{u_{**}+\epsilon}-\sqrt{u})^{2}\limsup_{N\to\infty}\frac{1}{N^{d-2}}\mathcal{E}_{\mathbb{Z}^{d}}(h_{N},h_{N}).
\end{equation}
Then, by Proposition \ref{prop:limsup}, taking consecutively the
limits $\eta\to0$, $r_U \to \infty$, and $\delta\to0$, and we obtain
\begin{equation}\label{5.3}
\liminf_{N\to\infty}\frac{1}{N^{d-2}}\log(\mathbb{P}_{u}[A_{N}])\geq- \dis\frac{1}{d} \;(\sqrt{u_{**}+\epsilon}-\sqrt{u})^{2}\mathrm{cap}_{\mathbb{R}^{d}}(K).
\end{equation}

\medskip\n
Finally, by taking $\epsilon\to0$ we obtain (\ref{eq:mainthm}) as
desired. 
\end{proof}

\begin{remark}\label{endremark} \rm ~

\medskip\n
1) It is an important question whether Theorem \ref{thm:lowerbound} can be complemented by a matching asymptotic upper bound, say when $K$ is a smooth
compact set. In view of Theorems 6.2 and 6.4 of \cite{LiSzni13} (see also Remark 6.5 2) of \cite{LiSzni13}), this would
indicate that the large deviations of the occupation-time profile of random interlacements,
insulating $K$ by values $u'$ of the local field (with $u'$ corresponding
to a non-percolative behaviour of $\mathcal{V}^{u'}$) capture the
main mechanism underlying the disconnection of a macroscopic body,
in the percolative regime of the vacant set.

\medskip\n
2) As $u\to0$, the right-hand side of (\ref{eq:mainthm})
tends to the finite limit $-\frac{u_{**}}{d}\mathrm{cap}(K)$.
One may wonder whether this limiting procedure retains any pertinence for the
study of the disconnection of the macroscopic body $K_{N}$ by a simple
random walk trajectory? For instance, does one have
\begin{equation}\label{5.4}
\liminf_{N\to\infty}\frac{1}{N^{d-2}}\log P_{0}\Big[\Big\{K_{N}\overset{Range\{(X_{t})_{t\geq0}\}^c}{\mbox{\Large $\longleftrightarrow$} \hspace{-3.5ex}/}\infty\Big\}\Big]\geq-\frac{u_{**}}{d}\mathrm{cap}_{\mathbb{R}^{d}}(K) ~?
\end{equation}
\hfill $\square$
\end{remark}

\end{document}